\documentclass[11pt,reqno]{amsart}

\usepackage{amsmath,amssymb,amsthm,mathtools,bm}
\usepackage{enumitem}
\usepackage{xcolor}
\usepackage{microtype}
\usepackage[
    colorlinks=true,
    linkcolor=blue!55!black,
    citecolor=blue!55!black,
    urlcolor=blue!55!black
]{hyperref}

\allowdisplaybreaks
\numberwithin{equation}{section}

\theoremstyle{plain}
\newtheorem{theorem}{Theorem}[section]
\newtheorem{lemma}[theorem]{Lemma}
\newtheorem{proposition}[theorem]{Proposition}
\newtheorem{corollary}[theorem]{Corollary}

\theoremstyle{definition}

\newtheorem*{counterexample}{Counterexample}

\theoremstyle{remark}
\newtheorem{remark}[theorem]{Remark}

\newcommand{\tr}{\operatorname{tr}}
\newcommand{\diver}{\operatorname{div}}
\newcommand{\dd}{\,\mathrm{d}}
\newcommand{\cG}{\mathcal{G}}

\title[Interior Hessian Estimates for Quotient Equations]
{Interior Hessian Estimates for Semi-convex Solutions of the
$\sigma_2/\sigma_1$ Equation with Lipschitz Right-Hand Sides}

\author{Ke Ji}
\address{
School of Mathematics and Statistics, Xi’an Jiaotong University, Xi’an, 710049, PR China
}
\email{jkxjtu123@stu.xjtu.edu.cn}

\author{Lichun Liang}
\address{
School of Mathematical Sciences, Chongqing Normal University, Chongqing, 401331, PR China
}
\email{lianglichun@cqnu.edu.cn}

\thanks{Corresponding author: Lichun Liang.}



\date{}

\begin{document}

\begin{abstract}
Let $n\ge2$ and let $u$ be a smooth 2-convex and semi-convex solution of
\[ \frac{\sigma _2(D^2u)}{\sigma _1(D^2u)}=f(x).
\]
We prove an interior Hessian estimate depending on the Lipschitz norm of $f$.
 The proof combines the integral approach of
Chen--Jian--Zhou with the algebraic reduction of the quotient equation to a
$\sigma _2$ structure.  The main new point is a shifted algebraic inequality that yields a shifted trace  Jacobi inequality in divergence form for $\log(\Delta u+a)$. We work with the linearized operator $G=(\Delta u-f)I-D^2u$ of the equivalent equation $\sigma_2(D^2u)=f \Delta u$. The almost divergence-free identity \(\partial_iG_{ij}=-f_j\) enables us to control the \(\Delta f\) term by integration by parts solely in terms of the Lipschitz norm of \(f\).  A Legendre--Lewy transformation converts the resulting degenerate divergence-form equation into a uniformly elliptic one. The estimate then follows from a mean-value inequality together with a weighted energy argument. As an application, in dimension two we obtain interior $C^2$ regularity for convex viscosity solutions with positive Lipschitz right-hand side. Moreover, our counterexamples show that the Lipschitz regularity required of the right-hand side is optimal.
\end{abstract}

\maketitle


\section{Introduction}
\noindent

This paper is devoted to obtain interior $C^2$ estimates for the Hessian quotient equation
\begin{equation}\label{eq:main-equation}
\left(\frac{\sigma_2}{\sigma_1}\right)(D^2u)
=
\left(\frac{\sigma_2}{\sigma_1}\right)
\bigl(\lambda(D^2u)\bigr)
=f(x),
\end{equation}
where $D^2u$ is the Hessian of $u$,
$\lambda=(\lambda_1,\ldots,\lambda_n)$ are the eigenvalues of $D^2u$,
\[
\sigma_k(\lambda)
=
\sum_{1\le i_1<\cdots<i_k\le n}
\lambda_{i_1}\cdots\lambda_{i_k}
\]
is the $k$-th elementary symmetric function defined on the $k$-th
G{\aa}rding's cone
\[
\Gamma_k
=
\left\{
\lambda\in\mathbb{R}^n
\,\middle|\,
\sigma_j(\lambda)>0,\quad j=1,\ldots,k
\right\},
\]
and $f$ is a prescribed positive function. A function $u\in C^2$ is called $k$-convex if $\lambda(D^2u)\in \Gamma_k$.

It is known that a standard pointwise argument differentiates the fully nonlinear equation twice and estimates the resulting term involving $D^2f$, thereby yielding an estimate whose constant depends on the $C^{1,1}$ or $C^2$ norm of $f$. 
We ask whether the second derivatives of $f$ can be removed  from the interior Hessian estimate for the Hessian quotient equation (\ref{eq:main-equation}). The main result gives an affirmative answer for semi-convex solutions of (\ref{eq:main-equation}).

\begin{theorem}\label{thm:main}
Let $n\ge2$ and $f\in C^2(B_{10})$ with $\inf_{B_{10}}f>0$. Assume that $u\in C^4(B_{10})$ be a 2-convex solution of \eqref{eq:main-equation} with $D^2u\geq -KI$ for some constant $K>0$.
Then we have 
\begin{equation}\label{eq:main-estimate}
|D^2u(0)|
 \le C(n,K,\inf_{B_{10}}f,\|f\|_{C^{0,1}(B_{10})},\|u\|_{L^\infty(B_{10})}).
\end{equation}
\end{theorem}

The $C^2$ assumption on $f$ is used only to justify classical
differentiation.  Thus \eqref{eq:main-estimate} is stable along smooth semi-convex
solutions whose right-hand sides have a uniform positive lower bound and a
uniform Lipschitz norm.  Theorem \ref{thm:main} is a smooth semi-convex solution's \emph{a priori} estimate. It does not assert that any semi-convex viscosity solution of \eqref{eq:main-equation} admits the regularity for the reason that the standard solvability of Dirichlet problem  for Hessian quotient equations in \cite{CNS1985} only produces smooth $2$-convex solutions, to which the existing $C^2$ estimates for semi-convex solutions do not apply in dimensions $n\ge3$. 

However, we can claim the regularity for convex viscosity solutions of \eqref{eq:main-equation} in the dimension $n=2$ by directly approximation.

\begin{theorem}\label{thm1}
Let $n=2$. Assume that $u$ is a convex viscosity solution to equation \eqref{eq:main-equation}  in $B_{20}$ with $f\in C^{0,1}(B_{20})$ and $\inf_{B_{20}} f>0$. Then $u\in C^2(B_{10})$ and 
$$\|u\|_{C^2(B_{10})}\leq C(n,\inf_{B_{20}}f,\|f\|_{C^{0,1}(B_{20})},\|u\|_{L^\infty(B_{20})}).$$
\end{theorem}

\begin{remark}
As a by-product, our method also recovers the regularity of almost-convex viscosity solutions and the a priori estimates for semi-convex solutions of \(\sigma_2(D^2u)=1\), previously established by  Shankar-Yuan \cite{ShankarYuanAlmostConvex} and Shankar-Yuan \cite{ShankarYuanSemiconvex}, respectively.
\end{remark}

For $n\geq 2$, we construct counterexamples showing that \eqref{eq:main-equation} admits non-smooth convex viscosity solutions without the assumption
$f\in C^{0,1}$.
\begin{theorem}\label{thm2}
For $0<\alpha<1$ and $n\geq 2$, let 
\begin{equation*}
u_\alpha(x)
 =\frac{|x_1|^{2-\alpha}}{(1-\alpha)(2-\alpha)}
  +\frac12\sum_{i=2}^n x_i^2\ \ \ \mbox{in} \ \ \ B_2.
\end{equation*}
Then $u_\alpha$ is a convex viscosity solution of
\begin{equation}\label{eq:quotient}
 \frac{\sigma_2(D^2u_\alpha)}{\sigma_1(D^2u_\alpha)}
 = f_\alpha(x),
 \qquad
 f_\alpha(x)
 =\frac{(n-1)+\binom{n-1}{2}|x_1|^\alpha }{1+(n-1)|x_1|^\alpha }.
\end{equation}
Moreover,
\[ f_\alpha\in C^{0,\alpha}(B_2),
 \qquad
 \inf_{B_2}f_\alpha\geq\frac{n-1}{1+(n-1)2^\alpha}>0,\]
whereas $f_\alpha$ is not Lipschitz near the origin and
\[ u_\alpha\in  C^{1,1-\alpha}_{\mathrm{loc}}(B_2).\]
\end{theorem}
In all, the Lipschitz exponent in an interior $C^2$ regularity theorem for convex solutions of the $\sigma_2/\sigma_1$ equation cannot be replaced by any H\"older exponent
$\alpha<1$.

Interior $C^2$ estimates have been a central theme in the field of fully nonlinear
elliptic equations since Heinz's \cite{Heinz} work on the two-dimensional Monge--Amp\`ere
equation.  Unfortunately, Pogorelov \cite{Pogorelov}  constructed
singular convex solutions of the Monge--Amp\`ere equation in dimensions
$n\ge3$ and Urbas \cite{Urbas} produced analogous counterexamples for
the $k$-Hessian equation for $k\geq 3$.  Consequently, interior $C^2$ estimates for the $2$-Hessian equation become an important problem worthy of further investigation.

For the $2$-Hessian equation, the first  interior $C^2$ estimate is
the two-dimensional result of Heinz.  Warren-Yuan \cite{WarrenYuan} used the special Lagrangian structure  to obtain the interior $C^2$ estimate for $\sigma_2(D^2u)=1$ in dimension three.  Qiu \cite{QiuHessian,QiuCurvature} subsequently treated $C^{1,1}$ right-hand sides in dimension three for both the $2$-Hessian equation and the prescribed scalar curvature equation for graphs. In dimension four, Shankar-Yuan \cite{ShankarYuanFour} proved the interior $C^2$ estimate for $\sigma_2(D^2u)=1$ by combining an almost Jacobi inequality with a doubling argument; their method also gives a new proof in dimension three. Fan \cite{Fan} extended the four-dimensional result to
$C^{1,1}$ right-hand sides. In high dimensions $n\geq 5$, however, all known interior $C^2$ estimates require some  convexity assumption.  McGonagle-Song-Yuan \cite{McGonagleSongYuan} obtained interior $C^2$ estimates for almost convex solutions of  $\sigma_2(D^2u)=1$ by compactness.  Shankar-Yuan \cite{ShankarYuanSemiconvex} providing an integral approach to interior $C^2$ estimates for 
semi-convex solutions of $\sigma_2(D^2u)=1$. Also, under a dynamic semi-convexity condition, Shankar-Yuan \cite{ShankarYuanFour} obtained the interior $C^2$ estimate for  $\sigma_2(D^2u)=1$. For $C^{1,1}$ right-hand sides, the interior $C^2$ estimate was established by Guan-Qiu \cite{GuanQiu}  under the structural condition
$\sigma_3(D^2u)\ge-A$. Moreover, Mooney \cite{Mooney} established strict
$2$-convexity for convex solutions of  $\sigma_2(D^2u)=1$, thereby establishing the regularity for convex viscosity solutions.

Now we pay attention to Hessian quotient equations
\begin{equation}\label{eq11}
  \left(\frac{\sigma_k}{\sigma_l}\right)(D^2u)=f.
\end{equation}
In recent years, substantial progress on interior $C^2$ estimates for Hessian quotient equations have emerged. The main earlier case was $k=3$ and $l=1$.  For $f=1$, interior $C^2$ estimates in dimensions $3$ and $4$ were  treated by Chen--Warren--Yuan \cite{ChenWarrenYuan} and Wang--Yuan \cite{WangYuanCritical} through special Lagrangian structure of the equation.  
Zhou \cite{Zhou} handled $C^{0,1}$ right-hand sides in dimensions three and four by the twisted special Lagrangian equation and Lu \cite{LuDimensionThree} later used a Jacobi inequality and the Legendre transform to give an alternative proof of the three-dimensional result with $C^{1,1}$ right-hand sides. For general Hessian quotient equations in high dimensions, the lack of a special Lagrangian structure makes the study of interior $C^2$ estimates particularly compelling. In this direction, Lu's \cite{LuGeneral} work made a significant contribution to identifying which Hessian quotient equations admit interior $C^2$ estimates, establishing a sharp dichotomy for a broad class of such equations.
In this direction, through counterexamples, Lu \cite{LuGeneral} made a significant contribution to identifying the range of Hessian quotient equations admitting interior $C^2$ estimates. Precisely, for $k=n$ and $l=n-1, n-2$, Lu \cite{LuGeneral} established  interior $C^2$ estimates for convex solutions of (\ref{eq11}) with $C^{1,1}$ right-hand sides. Moreover, he constructed singular convex viscosity solutions with smooth  right-hand sides for $k-l\geq 3$.  Consequently, the remaining open cases are $2\leq k\leq n$ and  $l=k-1,k-2$. In particular, for $k=2$ and $l=1$, when $f=1$, Lu-Sroka \cite{LuSroka}  discovered the shift identities from $\sigma_2/\sigma_1$ into $\sigma_2$, which yields interior $C^2$ estimates for 2-convex solutions of $\sigma_2(D^2u)=1$ in $n=3$ and $n=4$. Now Li-Wu \cite{LiWuQuadratic} established interior $C^2$ estimates for 2-convex solutions of $\sigma_2(D^2u)=f$ with $f\in C^{0,1}$ in all dimensions $n\geq2$. Jiao-Sui \cite{JiaoSui} considered $\sigma_2/\sigma_1$ with $C^{1,1}$ right-hand sides to obtain interior $C^2$ estimates of $2$-convex solutions in $n=3$ and of semi-convex solutions in $n\geq 4$. When $k=3$ and $l=1,2$ and $f=1$, Mei-Yan \cite{MeiYan2026} established interior $C^2$ estimates for semi-convex solutions in $n\geq 3$. In two nearly simultaneous works on interior $C^2$ estimates for convex solutions of \eqref{eq11} with $C^{1,1}$ right-hand sides, Tsai \cite{TsaiConcavity} treated \(l=k-1\) for \(2\leq k\leq n\), while Li-Wu \cite{LiWuConcavity} completed all cases  \(l=k-1,k-2\) for \(2\leq k\leq n\). Furthermore, Dong-Zhang \cite{DongZhang} obtained interior $C^2$ estimates for semi-convex solutions of \eqref{eq11} with $C^{1,1}$ right-hand sides for all case  \(2\leq k\leq n\) and \(l=k-1,k-2\).

However,  it is common for Hessian type equations  that these interior $C^2$ estimates  depend  on the $C^k$ norm of the right-hand side. So this leads to a question:

\vspace{5pt}
\textbf{Q.} What is the minimal regularity required of the right-hand side for an interior Hessian estimate ?
\vspace{5pt}

For the Monge--Amp\`ere equation, this question has been answered by Caffarelli \cite{Caffarelli1990} and Jian--Wang \cite{JianWang2007}, which yields interior $C^{2,\alpha}$ estimates under  H\"older regularity assumptions on the right-hand side.
For the $\sigma_2$ equation, some progress has already been made toward answering this question. In a pioneering work, Zhou \cite{Zhou} established the regularity and interior $C^2$ estimates for viscosity solutions of $\sigma_2(D^2u)=f$ with $f\in C^{0,1}$ in dimension three by exploiting its connection with the twisted special Lagrangian equation. 
Subsequently, Chen-Jian-Zhou \cite{ChenJianZhou} established interior $C^2$ estimates in all dimensions for convex solutions of $\sigma_2(D^2u)=f$ with Lipschitz right-hand side $f$. Following this work, Chen-Jian-Tu-Zhou \cite{ChenJianTuZhou} established interior regularity for viscosity solutions of $\sigma_2(D^2u)=f$ with Lipschitz right-hand side $f$.

In this paper, our main aim is to giving an  answer for Question \textbf{Q} for semi-convex solutions of Hessian quotient equations with Lipschitz right-hand sides. To the best of our knowledge, this work represents the first attempt to address this problem for Hessian quotient equations and yields several interesting results.

We conclude the introduction with a brief overview of the main ideas underlying the proof of Theorem \ref{thm:main}. We follow a prevailing strategy developed in \cite{ChenJianZhou,ShankarYuanSemiconvex}, which comprises three key ingredients. Firstly, we establish a strong shifted Jacobi inequality for $\log (\Delta u+a)$ and reformulate it in divergence form. Secondly, using the Legendre transform together with a mean-value-type inequality, we control the pointwise value $\log (\Delta u(0)+a)$ by an appropriate integral quantity. Finally, we estimate the resulting integral terms through integration by parts.

In our investigation of this problem, we first resolved the case of convex solutions.
Our approach retains the key insight of Zhou \cite{Zhou} and Chen-Jian-Zhou \cite{ChenJianZhou}: the Jacobi inequality should preserve a divergence structure $\Delta f$, so that the term containing  $\Delta f$ can be bounded by  the Lipschitz norm of $f$. This consideration leads us to work with the Jacobi quantity $\Delta u$, rather than the largest eigenvalue of $D^2u$ commonly used in pointwise arguments. Once this goal is carried out, however, we encounter a further difficulty in that the linearized operator of $\sigma_2/\sigma_1$  is not divergence-free, which makes integration by parts particularly challenging. To overcome this difficulty, we instead consider the linearized operator $G_{ij}$ associated with the equivalent equation $\sigma_2(D^2u)=\sigma_1(D^2u) f$. Since $\partial_iG_{ij}=-f_j$ is almost divergence-free, so   integration by parts produces only terms involving first derivatives of $f$. Differentiating $\sigma_2(D^2u)=\sigma_1(D^2u) f$ twice naturally produces the term $\Delta u \Delta f$, while division by $\Delta u$ gives rise to derivatives of $\log\Delta u$ in the resulting Jacobi inequality.  It is worth emphasizing that the Jacobi inequality is the crucial step in the proof. In fact, its derivation reduces to a purely algebraic inequality. More precisely, we need to establish the following inequality:
\begin{equation*}
  Q_k(X):=X_k^2+3\sum_{i\ne k}X_i^2-\left(\sum_iX_i\right)^2\geq \frac{4}{3}\frac{g_k}{\sigma_1} \left(\sum_i X_i\right)^2-\frac{C}{f^2}\left(\sum_i g_iX_i\right)^2
\end{equation*}
for $X=(X_1,\ldots,X_n)$ and $ g_i=\sigma_1(\lambda)-f-\lambda_i$ $(i=1,\ldots,n)$. In this inequality, the coefficient \(4/3\) is crucial: it must be strictly greater than \(1\) in order to obtain the strong Jacobi inequality.

However, this inequality fails in the semi-convex setting.  Under the homogeneous constraint $\sum_i g_iX_i=0$, we give the following:
\begin{counterexample}
Let \(n=5\), \(\sigma_1=1\) and \[ \lambda
= \left(-\frac{1}{2},\frac{3}{8},\frac{3}{8},\frac{3}{8},\frac{3}{8}\right).
\]
Then\[\sigma_2(\lambda)=\frac{3}{32},\qquad
f=\frac{3}{32},\qquad g_1=\frac{45}{32},
\qquad g_2=\cdots=g_5=\frac{17}{32}.
\]
For \(k=1\), minimize \(Q_1\) subject to
\[ \sum_i X_i=1\qquad\text{and}\qquad
g\cdot X=0. \]
Symmetry gives
\[ X_2=\cdots=X_5=y,\qquad
X_1=x,\]
and
\[x+4y=1,\qquad
45x+68y=0.
\]
Thus
\[x=-\frac{17}{28},
\qquad
y=\frac{45}{112}.\]
A direct calculation gives
\[Q_1(X)=\frac{4095}{3136},
\qquad
\frac{Q_1(X)}
{g_1\left(\sum_i X_i\right)^2}
=\frac{13}{14}<1.\]
\end{counterexample}
Consequently, we cannot obtain an inequality
\begin{equation}\label{eq31}
  Q_1(X) \geq (1+\theta)\frac{g_1}{\sigma_1} \left(\sum_i X_i\right)^2,
\qquad \theta>0.
\end{equation}
To establish inequality \eqref{eq31}, we distinguish two cases according to the size of \(\sigma_1\). When \(\sigma_1\) is sufficiently large, normalization makes the semi-convexity lower bound tend to zero, so the normalized configuration approaches the convex regime; in this case, the coefficient in \eqref{eq31} can be taken strictly greater than \(1\). When \(\sigma_1\) remains bounded, inequality \eqref{eq31} follows from a compactness argument, although the resulting coefficient may be less than \(1\). This difficulty can be overcome by introducing a suitable shifted quantity \(a\); more precisely, we can choose $a>0$ such that 
\[ c_0\frac{1}{\sigma_1}
\geq (1+\theta)\frac{1}{\sigma_1+a}
\]
for some constant $0<c_0\leq 1$.

The remainder of the paper is organized as follows. Section~2 develops the algebraic structure of the quotient equation, including the shifted identity and the spectral estimates for the linearized operator \(G\). In Section~3, we prove the key shifted algebraic inequality and use it to derive both the strong and the weak trace--Jacobi inequalities for \(\log(\Delta u+a)\), together with a Caccioppoli-type energy estimate. 
Section~4 uses the Legendre--Lewy transformation to obtain a uniformly elliptic setting in which a mean-value inequality controls \(\log(\Delta u(0)+a)\). Section~5 establishes the weighted integral estimate that completes the proof of Theorem~\ref{thm:main}.
Section~6 presents singular solutions for \eqref{eq:main-equation} with $f\in C^{0,\alpha}$ for $0<\alpha<1$.

\section{Algebraic structure of the quotient equation}
\noindent

Throughout the paper, we denote
\[\sigma_{k;i_1\cdots i_l}(\lambda):=\left.\sigma_k(\lambda)
\right|_{\lambda_{i_1}=\cdots=\lambda_{i_l}=0}\]
for \(1\leq k\leq n\).

The Newton--Maclaurin inequality gives
\begin{equation*}
 \sigma _2(D^2u)\leq \frac{n-1}{2n}(\Delta u)^2,
\end{equation*}
which leads to 
\begin{equation}\label{eq:s-lower}
\Delta u  \geq \frac{2n}{n-1}f.
\end{equation}
The linearized operator of $\sigma _2(D^2u)-f \sigma_1(D^2u)$ is
\begin{equation*}
 G_{ij}=\frac{\partial \sigma_2 (A)}{\partial A_{ij}}=\tr A \delta_{ij}-A_{ij}=(\Delta u-f)\delta_{ij}-u_{ij},
\end{equation*}
where \[ A=D^2u-\frac{f}{n-1}I.\]

Motivated by  the ideas from \cite{LuSroka}, we establish the following shifted lemma.

\begin{lemma}\label{lem:shift}
Let $u\in C^2(B_{10})$ be a $2$-convex solution of (\ref{eq:main-equation}) and $f\in C(B_{10})$ be positive. Then we have $A=D^2u-\frac{f}{n-1}I \in\Gamma _2$, i.e.,
\begin{equation}\label{eq:shift-id}
 \sigma _2(A)=\frac{n}{2(n-1)}f^2,
 \qquad \frac{n}{n-1}f\le \tr A \le \Delta u,
\end{equation}
and $G=(G_{ij})=\tr A I-A$ is positive definite.  Moreover, we have 
\begin{equation}\label{eq:G-basic}
  G \ge \frac{\sigma_2(A)}{\tr A}I,
 \qquad
 \tr(G^{-1})\le 2(n-1)\frac{\Delta u}{f^2},
 \qquad
 D_iG_{ij}=-f_j.
\end{equation}
\end{lemma}

\begin{proof}
For every symmetric matric $H$ and  $c\in \mathbb{R}$, we know that 
\[
 \sigma _2(H-cI)=\sigma _2(H)-(n-1)c\sigma _1(H)
 +\binom n2c^2.
\]
Taking $H=D^2u$ and $c=f/(n-1)$ proves the first identity in (\ref{eq:shift-id}).  
From \eqref{eq:s-lower}, it follows that 
$$\tr A \geq \frac{n}{n-1}f.$$ Hence $A\in\Gamma _2$.

Assume that  $A$ is diagonal and denote  its eigenvalues by $\mu_i$ ($i=1,\ldots,n$).  Then
$G_{ii}=\tr A-\mu_i=\sum_{j\ne i}\mu_j$ and we have 
\begin{equation*}
  \begin{split}
     \tr A G_{ii}-\sigma _2(A)=& (\mu_i+\sigma_{1;i}(\mu))\sigma_{1;i}(\mu)-(\mu_i \sigma_{1;i}(\mu)+\sigma_{2;i}(\mu)) \\
      =& \sigma_{1;i}^2(\mu)-\sigma_{2;i}(\mu)\\
      =& \sigma_{1;i}^2(\mu)-\frac{1}{2}\left(\sigma_{1;i}^2(\mu)-\sum_{j\ne i}\mu_j^2\right)\\
      =&\frac12\left(\sigma_{1;i}^2(\mu)+\sum_{j\ne i}\mu_j^2\right)\geq 0.
  \end{split}
\end{equation*}
This proves the first estimate in \eqref{eq:G-basic}. That is, we have 
$$G_{ii}\ge \frac{\sigma_2(A)}{\tr A}=\frac{n}{2(n-1)}\frac{f^2}{\tr A}>0,$$ 
and therefore $$\frac{1}{G_{ii}}\leq \frac{2(n-1)}{n}\frac{\tr A}{f^2},$$
which implies that 
$$\tr(G^{-1})\le 2(n-1)\frac{\tr A}{f^2}\leq 2(n-1)\frac{\Delta u}{f^2} .$$
  Finally, we obtain
\[
 D_i G_{ij}=(D_i \Delta u-f_i)\delta_{ij}-u_{iij}= \Delta u_j-f_j-u_{iij}
=-f_j.\]
\end{proof}

\begin{lemma}[Spectral structure]\label{lem:spectral}
Let $n\geq 2$ and $K\geq 0$. Assume that $\lambda\in \Gamma_2$ satisfies $\lambda_1\ge\cdots\ge\lambda_n\geq -K$ and $\frac{\sigma_2(\lambda)}{\sigma_1(\lambda)}=f$ with $\underline{f}\leq f\leq \overline{f}$ for two constants $\overline{f}\geq \underline{f}>0$.
Then we have 
$$\underline{f}\leq \sigma_{1;1}(\lambda)\leq \widetilde{C},$$
$$-K\leq \lambda_i\leq \widetilde{C}+(n-2)K,
\qquad i=2,\ldots n,$$
where $$\widetilde{C}= (n-1) \left(\bar{f}+(n-2)K+\frac{(n-1)(n-2)}{2}K^2 \right) \left(2+\frac{1}{\bar{f}}\right)+2\bar{f}.$$
Furthermore, for $g_{i}=\sigma_1(\lambda)-\lambda_i-f$ ($i\geq 1$), we have the following estimates:
$$\frac{1}{K+1+\lambda_1}\frac{1}{4(n-1)}\underline{f}^2\leq g_{1}\leq \frac{(K+1)\sigma_{1;1}(\lambda)+\overline{f}\sigma_{1;1}(\lambda)+|\sigma_{2;1}(\lambda)|}{(K+1+\lambda_1)}$$
and \[\begin{aligned}
&\min\left\{ \frac{1}{4}, \frac{1}{36(\widetilde{C}+(n-2)K+K+1)^2}
 \frac{1}{n-1}\underline{f}^2 \right\}\\
&\qquad \qquad \qquad \qquad \qquad \leq \frac{g_i}{K+1+\lambda_1}
\leq \sigma_{1;1}(\lambda)+\overline{f}+1,\ \ \ i=2,\ldots,n.
\end{aligned}
\]
\end{lemma}

\begin{proof}
The equation reads
\[
 f=\frac{\lambda_1 \sigma_{1;1}(\lambda)+ \sigma_{2;1}(\lambda)}{\lambda_1+\sigma_{1;1}(\lambda)}.
\]
That is,
\begin{equation}\label{eq12}
\lambda_1(\sigma_{1;1}(\lambda)-f)=f\sigma_{1;1}(\lambda)-\sigma_{2;1}(\lambda).
\end{equation}
Since
\[
g_{1}=\sigma_1-\lambda_1-f=\sigma_{1;1}(\lambda)-f,
\]
equation (\ref{eq12}) can equivalently be written as
\begin{equation}\label{eq15}
\lambda_1g_{1}=f\sigma_{1;1}(\lambda)-\sigma_{2;1}(\lambda).
\end{equation}
\noindent\textbf{Estimate $\sigma_{1;1}(\lambda)$.} 
By $g_{1}>0$ in Lemma \ref{lem:shift}, we obviously have 
\[
 \sigma_{1;1}(\lambda)-f>0.
\]
For each $i=2,\ldots,n$, the semi-convexity assumption gives
\[
\nu_i:=\lambda_i+K\geq 0.
\]
Since there are $n-1$ such variables, we have
\[
\sum_{i=2}^{n}\nu_i=\sigma_{1;1}(\lambda)+(n-1)K.
\]
By the translation formula for the elementary symmetric functions,
\[
\begin{aligned}
\sigma_{2;1}(\lambda)
&=\sigma_2(\nu-K\mathbf{1})\\
&=\sigma_2(\nu)
 -(n-2)K\sigma_1(\nu)
 +\binom{n-1}{2}K^2.
\end{aligned}
\]
Since $\nu_i\geq 0$ for all $i=2,\ldots,n$, it follows that
\[
\sigma_2(\nu)\geq 0.
\]
Therefore,
\[
\begin{aligned}
\sigma_{2;1}(\lambda)
&\geq  -(n-2)K\bigl(\sigma_{1;1}(\lambda)+(n-1)K\bigr) +\frac{(n-1)(n-2)}{2}K^2\\
&=  -(n-2)K\sigma_{1;1}(\lambda)  -\frac{(n-1)(n-2)}{2}K^2.
\end{aligned}
\]
It follows that
\begin{equation}\label{eq13}
\begin{aligned} \lambda_1(\sigma_{1;1}(\lambda)-f)
&=f\sigma_{1;1}(\lambda)-\sigma_{2;1}(\lambda)\\
&\leq f\sigma_{1;1}(\lambda)+(n-2)K\sigma_{1;1}(\lambda)  +\frac{(n-1)(n-2)}{2}K^2\\
&\leq C_{n,K,\bar{f}}(\sigma_{1;1}(\lambda)+1)
\end{aligned}
\end{equation}
for some constant $$C_{n,K,\bar{f}}=\bar{f}+(n-2)K+\frac{(n-1)(n-2)}{2}K^2.$$
On the other hand, since $\lambda_i\leq \lambda_1$ for $i\geq 2$, we obtain
\begin{equation*}
(n-1)\lambda_1\geq \sigma_{1;1}(\lambda).
\end{equation*}
Substituting this into (\ref{eq13}) yields
\begin{equation}\label{eq14}
\frac{\sigma_{1;1}(\lambda)}{n-1}(\sigma_{1;1}(\lambda)-f)\leq C_{n,K,\bar{f}} (\sigma_{1;1}(\lambda)+1).
\end{equation}
If $\sigma_{1;1}(\lambda)\geq 2\bar{f}$, then
\[
\sigma_{1;1}(\lambda)-f\geq \sigma_{1;1}(\lambda)-\bar{f}\geq \frac{\sigma_{1;1}(\lambda)}{2}.
\]
Consequently, (\ref{eq14}) implies that
\[
\frac{\sigma_{1;1}^2(\lambda)}{2(n-1)}\leq C_{n,K,\bar{f}}  (\sigma_{1;1}(\lambda)+1),
\]
and therefore $$\sigma_{1;1}(\lambda)\leq 2(n-1)C_{n,K,\bar{f}} \left(1+\frac{1}{\sigma_{1;1}(\lambda)}\right)\leq 2(n-1)C_{n,K,\bar{f}} \left(1+\frac{1}{2\bar{f}}\right).$$
 If $\sigma_{1;1}(\lambda)<2\bar{f}$, then $\sigma_{1;1}(\lambda)$ is already uniformly bounded. It follows that
\[\sigma_{1;1}(\lambda)\leq 2(n-1)C_{n,K,\bar{f}} \left(1+\frac{1}{2\bar{f}}\right)+2\bar{f}.
\]

\noindent\textbf{Estimate $\lambda_i$ for $i\geq 2$.} 
 The lower bound follows immediately  from the semi-convexity
\[\lambda_i\geq -K.\]
To obtain the upper bound, fixing $i\geq 2$, 
we have
\[ \lambda_i = \sigma_{1;1}(\lambda)-\sum_{\substack{j=2\\j\neq i}}^{n}\lambda_j.
\]
From $\lambda_j\geq -K$ for every $j$, it follows that
\[ -\sum_{\substack{j=2\\j\neq i}}^{n}\lambda_j
\leq (n-2)K.\]
Hence 
\[\lambda_i\leq \sigma_{1;1}(\lambda)+(n-2)K. \]
Combining this with the estimate for $\sigma_{1;1}(\lambda)$, we conclude that
\begin{equation}\label{eq18}
-K\leq \lambda_i\leq 2(n-1)C_{n,K,\bar{f}} \left(1+\frac{1}{2\bar{f}}\right)+2\bar{f}+(n-2)K,
\qquad i=2,\ldots,n.
\end{equation}
\noindent\textbf{Estimate $g_{1}$.} 
 Recall $g_{1}=\sigma_{1;1}(\lambda)-f$ and $\lambda_1g_{1}=f\sigma_{1;1}(\lambda)-\sigma_{2;1}(\lambda)$. From the bound for $\sigma_{1;1}(\lambda)$ and $\lambda_i$ ($i\geq 2$), we see that 
 $$g_{1}\leq \sigma_{1;1}(\lambda) \qquad\mbox{and}\qquad \lambda_1g_{1}\leq \overline{f}\sigma_{1;1}(\lambda)+|\sigma_{2;1}(\lambda)|$$
which implies that 
$$(K+1+\lambda_1)g_{1}\leq (K+1)\sigma_{1;1}(\lambda)+\overline{f}\sigma_{1;1}(\lambda)+|\sigma_{2;1}(\lambda)|.$$
Next, to prove the lower bound for $(K+1+\lambda_1)g_{1}$, we need to prove that
\[ \lambda_1g_{1}\geq C \]
for some constant $C>0$ depending only on $n$ and $\underline{f}$. 
Set
\[
\alpha:=\frac{n-2}{2(n-1)}<1.
\]
Applying the Cauchy--Schwarz inequality to the $n-1$ real numbers
$\lambda_2,\ldots,\lambda_n$, we obtain
\[ \sum_{i=2}^{n}\lambda_i^2
\geq \frac{(\sigma_{1;1}(\lambda))^2}{n-1}.
\]
Therefore, we have 
\begin{equation}\label{eq16}
\begin{aligned}
\sigma_{2;1}(\lambda)
&= \frac{1}{2}\left( \sigma_{1;1}^2(\lambda)-\sum_{i=2}^{n}\lambda_i^2 \right)\\
&\leq \frac{1}{2}\left(\sigma_{1;1}^2(\lambda)-\frac{\sigma_{1;1}^2(\lambda)}{n-1}\right)=\alpha \sigma_{1;1}^2(\lambda).
\end{aligned}
\end{equation}
Choose a sufficiently small constant $\delta>0$ depending only on
$n$ and $\underline{f}$ such that
\begin{equation}\label{eq17}
\alpha\delta\leq \frac{1-\alpha}{2}\underline{f}.
\end{equation}
We divide the remaining proof into two cases.

\noindent\textbf{Case 1: $g_{1}\geq\delta$.}

The Newton--Maclaurin inequality yields
\[ \sigma_2(\lambda)\leq\frac{n-1}{2n}\sigma_1^2(\lambda).\]
Since $\sigma_1(\lambda)>0$, dividing both sides by
$\sigma_1(\lambda)$ gives
\[f=\frac{\sigma_2(\lambda)}{\sigma_1(\lambda)}\leq\frac{n-1}{2n}\sigma_1(\lambda).\]
Moreover, since $\lambda_1$ is the largest eigenvalue, it is no smaller
than the average of all the eigenvalues:
\[\lambda_1\geq\frac{\sigma_1(\lambda)}{n}.\]
Consequently,
\begin{equation*}
\lambda_1\geq\frac{2f}{n-1}\geq\frac{2\underline{f}}{n-1}.
\end{equation*}
In this case, we have
\begin{equation*}
\lambda_1g_{1} \geq\frac{2\delta}{n-1}\underline{f}.
\end{equation*}

\noindent\textbf{Case 2: $0<g_{1}<\delta$.}

It is clear that 
\[ \sigma_{1;1}(\lambda)<f+\delta.\]
Recalling \eqref{eq15} and \eqref{eq16}, we have
\[ \begin{aligned}
\lambda_1g_{1}&=f\sigma_{1;1}(\lambda)-\sigma_{2;1}(\lambda)\\
&\geq f\sigma_{1;1}(\lambda)-\alpha \sigma_{1;1}^2(\lambda)\\
&=\sigma_{1;1}(\lambda)(f-\alpha \sigma_{1;1}(\lambda))\\
&\geq \sigma_{1;1}(\lambda) \bigl(f-\alpha(f+\delta)\bigr)\\
&=\sigma_{1;1}(\lambda)\bigl((1-\alpha)f-\alpha\delta\bigr).
\end{aligned}
\]
By  $f\geq \underline{f}$ and (\ref{eq17}), we see that 
\[ (1-\alpha)f-\alpha\delta\geq\frac{1-\alpha}{2}\underline{f}.\]
Moreover, since $ \sigma_{1;1}(\lambda)>f\geq \underline{f}$, we conclude that
\begin{equation*}
\lambda_1g_{1} \geq\frac{1-\alpha}{2}\underline{f}^2.
\end{equation*}
Combining the two cases, we obtain
\[ \lambda_1g_{1}
\geq \min\left\{\frac{2\delta}{n-1}\underline{f},\frac{1-\alpha}{2}\underline{f}^2\right\}.
\]
Since $g_{1}>0$, it follows that
\begin{equation}\label{eq20}
  (K+1+\lambda_1)g_{1}\geq\lambda_1g_{1}\geq \min\left\{\frac{2\delta}{n-1}\underline{f},\frac{1-\alpha}{2}\underline{f}^2\right\}.
\end{equation}
Choosing $\delta=\frac{1}{8}\underline{f}$ yields the desired estimate, i.e.,
$$ \min\left\{\frac{2\delta}{n-1}\underline{f},\frac{1-\alpha}{2}\underline{f}^2\right\}=\frac{1}{4(n-1)}\underline{f}^2.$$

\noindent\textbf{Estimate $g_{i}$ ($i\geq 2$).} 
For each $i\geq 2$, we have
\begin{equation}\label{eq19}
\begin{aligned}
g_{i}
&=\sigma_1-\lambda_i-f\\
&=\lambda_1+\sigma_{1;1}(\lambda)-\lambda_i-f\\
&=(\sigma_{1;1}(\lambda)-f)+\lambda_1-\lambda_i\\
&=g_{1}+\lambda_1-\lambda_i.
\end{aligned}
\end{equation}
Since $g_{1}\leq \sigma_{1;1}(\lambda)$ and $\lambda_i\geq -K$, we have the upper bound
\begin{equation*}
g_{i} \leq \sigma_{1;1}(\lambda)+\lambda_1+K\leq (\sigma_{1;1}(\lambda)+1)(K+1+\lambda_1),
\qquad i\geq 2.
\end{equation*}
To give the lower bound, by (\ref{eq18}), there exists a constant $\widetilde{M}=\widetilde{C}+(n-2)K>0$ depending only on $n$, $K$ and $\overline{f}$  such that 
$$-K\leq \lambda_i\leq \widetilde{M},\qquad i\geq 2.$$
\noindent\textbf{Case 1: $\lambda_1\geq 2(\widetilde{M}+K+1)$.}

By (\ref{eq19}) and  $\lambda_i\leq \widetilde{M}$, we have
\[ g_{i}\geq\lambda_1-\lambda_i
\geq\lambda_1-\widetilde{M}\geq
\frac{\lambda_1}{2}.\]
Moreover, the assumption of this case implies $\lambda_1\geq K+1$ and therefore
\[K+1+\lambda_1\leq 2\lambda_1.\]
So 
\begin{equation}\label{eq30}
g_{i} \geq
\frac{1}{4}(K+1+\lambda_1),
\qquad i\geq 2.
\end{equation}

\noindent\textbf{Case 2: $\lambda_1<2(\widetilde{M}+K+1)$.}
By (\ref{eq20}), we have 
\[g_{1}\geq\frac{1}{1+K+\lambda_1} \frac{1}{4(n-1)}\underline{f}^2.\]
In the present case,
\[K+1+\lambda_1 \leq 3(\widetilde{M}+K+1),\]
and therefore
\[ g_{1}\geq \frac{(K+1+\lambda_1)}{9(\widetilde{M}+K+1)^2} \frac{1}{4(n-1)}\underline{f}^2.
\]
Since $\lambda_1\geq\lambda_i$, (\ref{eq19}) gives
\[ g_{i}\geq g_{1},\]
which together with \eqref{eq30} yields the lower bound
$$ \min\left\{\frac{1}{4},\frac{1}{9(\widetilde{M}+K+1)^2} \frac{1}{4(n-1)}\underline{f}^2\right\}\leq \frac{g_{i}}{K+1+\lambda_1} .$$
It follows that we obtain the desired estimate for $g_i$.
\end{proof}

\section{Shifted trace Jacobi inequality}
\noindent

In order to establish the shifted trace Jacobi inequality for $\log (\Delta u+a)$ for some shifted quantity $a$, we prove the following shifted algebraic inequality 
 independent of the equation.

\begin{lemma}\label{lem:algebra-semiconvex}
Let $n\geq 2$ and $K\geq 0$. Suppose that $\lambda\in \Gamma_2$ satisfies $\lambda_1\ge\cdots\ge\lambda_n\geq -K$ and $\frac{\sigma_2(\lambda)}{\sigma_1(\lambda)}=f$ with $\underline{f}\leq f\leq \overline{f}$ for two constants $\overline{f}\geq \underline{f}>0$ and that 
$$g_i:=\sigma_1(\lambda)-f-\lambda_i,\ \ \ i=1,\ldots,n.$$
Let $T=(T_{ijk})$ be a completely symmetric three-tensor and define
\[
 t_k:=\sum_iT_{iik},
 \qquad
 q_k:=\sum_i g_iT_{iik},
\]
then there exist some constants $a>0$ and $C>0$ depending only on $n$, $K$, $\overline{f}$ and $\underline{f}$ such that 

\begin{equation*}
 \sum_{i,j,k}T_{ijk}^2-\sum_kt_k^2
 \ge  \frac{37}{36} \frac1{\sigma_1(\lambda)+a}\sum_k g_kt_k^2
 -\frac{C}{f^2}\sum_k q_k^2.
\end{equation*}
\end{lemma}

\begin{proof}
We divide the proof into several steps.

\medskip
\noindent
\textbf{Step 1. Elementary properties of $\sigma_1$, $f$ and $g_i$.}

From Section 2, we recall some results on $\sigma_1$, $f$ and $g$.
By (\ref{eq:s-lower}), we see that 
\begin{equation}\label{eq25}
  \sigma_1(\lambda) \ge\frac{2n}{n-1}f.
\end{equation}
Moreover, we have
\begin{equation*}
 \sum_i g_i = (n-1)\sigma_1(\lambda)-nf \ge \frac{n-1}{2}\sigma_1(\lambda).
\end{equation*}
In addition,  the Cauchy–Schwarz inequality yields
\begin{equation}\label{eq28}
   \sqrt{\sum_i g_i^2} \geq \frac{1}{\sqrt{n}}\sum_i g_i \geq \frac{n-1}{2}\frac{1}{\sqrt{n}}\sigma_1(\lambda).
\end{equation}
We set $$\mu=\lambda-\frac{f}{n-1}(1,\ldots,1).$$ Therefore, by Lemma \ref{lem:shift}, we conclude that 
\begin{equation}\label{eq24}
  g_i\geq \frac{n}{2(n-1)}\frac{f^2}{\sigma_1(\mu)}\geq \frac{n}{2(n-1)}\frac{f^2}{\sigma_1(\lambda)}.
\end{equation}
Finally, since
\[ 2f\sigma_1(\lambda)=2\sigma_2(\lambda)=\sigma_1^2(\lambda)-|\lambda|^2>0,\]
we have $|\lambda|<\sigma_1(\lambda)$. Hence we have the estimate 
\begin{equation}\label{eq:g-upper}
 0<g_i=\sigma_1(\lambda)-f-\lambda_i<2 \sigma_1(\lambda).
\end{equation}

\medskip
\noindent
\textbf{Step 2. Reduction to a one-direction quadratic form.}

Since $T$ is completely symmetric, we have 
\begin{equation}\label{eq:tensor-reduction-semiconvex}
 \sum_{i,j,k}T_{ijk}^2  \ge \sum_{k=1}^n \left( T_{kkk}^2+3\sum_{i\ne k}T_{iik}^2 \right).
\end{equation}
Indeed, \(T_{kkk}\) occurs once and each \(T_{iik}\) with \(i\ne k\) occurs three times, while the nonnegative terms with three distinct indices are omitted.

For any fixed $k$, we set
\[ X_i:=T_{iik},\qquad t:=\sum_iX_i,\qquad
 q:=\sum_i g_iX_i
\]
and also let $X=(X_1,\ldots,X_n)$ and $g=(g_1,\ldots,g_n)$.

Define a quadratic form
\begin{equation*}
 Q_k(X) := X_k^2+3\sum_{i\ne k}X_i^2 -\left(\sum_iX_i\right)^2.
\end{equation*}
So we need to prove
\begin{equation*}
 Q_k(X)  \ge \frac{37}{36}\frac{g_k}{\sigma_1(\lambda)+a}t^2 -Cf^{-2}q^2.
\end{equation*}

\medskip
\noindent
\textbf{Step 3. Exact minimization under the homogeneous constraint.}

We first consider the homogeneous constraint
\[
 q=\sum_i g_iX_i=0.
\]
Normalize $\lambda$ by setting
\[
 \widetilde\lambda_i:=\frac{\lambda_i}{\sigma_1(\lambda)},
 \qquad
 \widetilde f:=\frac{f}{\sigma_1(\lambda)},
 \qquad
 \widetilde g_i:=\frac{g_i}{\sigma_1(\lambda)}
 =1-\widetilde f-\widetilde\lambda_i.
\]
Then
\[
 \sum_i\widetilde\lambda_i=1,
 \qquad
 \sigma_2(\widetilde\lambda)=\widetilde f,
 \qquad
 \sum_i \widetilde g_i X_i=0.
\]
We also have
\begin{equation*}
 \sum_i\widetilde g_i=n-1-n\widetilde f.
\end{equation*}
Moreover,  from 
$$\sum_i\widetilde\lambda_i^2  = \left(\sum_i\widetilde\lambda_i\right)^2 -2\sigma_2(\widetilde\lambda) = 1-2\widetilde f,$$ we see that 
\begin{equation}\label{eq21}
  \begin{aligned}
\sum_i \widetilde g_i^2&=\sum_i(1-\widetilde f-\widetilde \lambda_i)^2\\
&=\sum_i\left[(1-\widetilde f)^2-2(1-\widetilde f)\widetilde \lambda_i+\widetilde \lambda_i^2\right]\\
&=n(1-\widetilde f)^2-2(1-\widetilde f)\sum_i\widetilde \lambda_i+\sum_i\widetilde \lambda_i^2\\
&=n-1-2n\widetilde f+n\widetilde f^2.
\end{aligned}
\end{equation}
Let $W_k=diag (w_1,w_2,\dots,w_n)$ be the diagonal matrix defined by
 \begin{equation*}
   w_i=
\begin{cases}
1, & i=k,\\
3, & i\neq k,
\end{cases}
 \end{equation*}
then we write
\[ Q_k(X)=XW_kX^{T}-t^2.\]
For fixed $t$, we minimize $XW_kX^{T}$ subject to
\[  \sum_iX_i=t,\qquad \sum_i\widetilde g_iX_i=0.
\]
Also, if we set 
\[
 A=
 \begin{pmatrix}
 1&\cdots&1\\
 \widetilde g_1&\cdots&\widetilde g_n
 \end{pmatrix},
 \qquad
 b=
 \begin{pmatrix}
 t\\0
 \end{pmatrix}.
\]
then the two constraint conditions can be written as 
 $$AX^T=b.$$ 
As a result, we will solve the problem
$$\min\left\{X W_k X^T : AX^T=b\right\}.$$
 From the well-known fact, it follows that the minimum is equal to 
 $$b^T\left(AW_k^{-1}A^T\right)^{-1}b.$$
 Set 
 $$M:=AW_k^{-1}A^T,$$
 i.e.,
 
 \begin{equation}\label{eq1}
   M=
\begin{pmatrix}
\displaystyle\sum_i \frac{1}{w_i}
&
\displaystyle\sum_i \frac{\widetilde g_i}{w_i}
\\[6pt]
\displaystyle\sum_i \frac{\widetilde g_i}{w_i}
&
\displaystyle\sum_i \frac{\widetilde g_i^2}{w_i}
\end{pmatrix}.
 \end{equation} 
For simplify, we introduce some notations 
\[  L:=\sum_i \widetilde g_i= n-1-n\widetilde f, \qquad x:=\widetilde g_k,
\]
and
\[ N:=n-1-2n\widetilde f+n\widetilde f^2+2x^2.\]
Since $w_k=1$ and $w_i=3$ for $i\neq k$, we have  
$$\sum_i \frac{1}{w_i}=\frac{1}{w_k}
+\sum_{i\neq k}\frac{1}{w_i}=1+\frac{n-1}{3}=
\frac{n+2}{3}.$$
Hence
\begin{equation*}
 \sum_i\frac{\widetilde g_i}{w_i}
= x+\frac{1}{3}\sum_{i\neq k}\widetilde g_i
=x+\frac{1}{3}(L-x)
=\frac{L+2x}{3}.
\end{equation*}
In view of (\ref{eq21}), we obtain 
 \begin{equation*}
   \begin{aligned}
\sum_i\frac{\widetilde g_i^2}{w_i}
&=
x^2+\frac{1}{3}\sum_{i\neq k}\widetilde g_i^2\\
&=
x ^2+\frac{1}{3}\left(n-1-2n\widetilde f+n\widetilde f^2-x^2\right)\\
&=
\frac{n-1-2n\widetilde f+n\widetilde f^2+2x^2}{3}=\frac{N}{3}.
\end{aligned}
 \end{equation*}
It follows that  
 \begin{equation*}
   M=\frac{1}{3}
\begin{pmatrix}
n+2 & L+2x\\
L+2x & N
\end{pmatrix}.
 \end{equation*}
A direct computation yields 
\begin{equation*}
  M^{-1}
=\frac{3}{D}
\begin{pmatrix}
N & -(L+2x)\\
-(L+2x) & n+2
\end{pmatrix},
\end{equation*}
which implies that 
\begin{equation*}
  \begin{aligned}
b^T M^{-1}b
&=\frac{3}{D}
\begin{pmatrix}
t \ \ 0
\end{pmatrix}
\begin{pmatrix}
N & -(L+2x)\\
-(L+2x) & n+2
\end{pmatrix}
\begin{pmatrix}
t\\
0
\end{pmatrix}\\
&=
\frac{3N}{D}t^2,
\end{aligned}
\end{equation*}
where $D:=9\det M=(n+2)N-(L+2x)^2.$

Recalling (\ref{eq1}), by the aid of  Cauchy–Schwarz inequality, we have 
$$\det M=\sum_i \left(\frac{1}{\sqrt{w_i}}\right)^2 \sum_i \left(\frac{\widetilde g_i}{\sqrt{w_i}}\right)^2-\sum_i \frac{1}{\sqrt{w_i}} \frac{\widetilde g_i}{\sqrt{w_i}}\sum_i \frac{1}{\sqrt{w_i}} \frac{\widetilde g_i}{\sqrt{w_i}}\geq 0.$$
If $\det M=0$, then all $\widetilde g_i$ are equal, i.e., $\widetilde g_i=c$ for some constant $c$. Hence the homogeneous constraint is $$0=\sum_i \widetilde g_iX_i=c \sum_i X_i=c t.$$
Due to $ \widetilde g_i>0$, we have $c>0$, which yields $$t=0.$$
Without loss of generality, we assume that $\det M>0$. We consider the quantity 
\begin{equation}\label{eq22}
  \min Q_k(X)=b^T M^{-1}b-t^2=\frac{3N-D}{D}t^2.
\end{equation}
We next verify that the coefficient $3N-D$ in \eqref{eq22} is
strictly positive. A direct calculation gives
\begin{equation}\label{eq:positive-numerator}
 3N-D  = n\widetilde f^2+4xL-2(n-3)x^2.
\end{equation}
Since $x$ and $L$ are strict positive, when $n=2$ and $3$, it is seen that $3N-D>0$. Next we consider the case $n\geq 4$.
Since
\[ |\widetilde\lambda_k|
 \le|\widetilde\lambda|
 =\sqrt{1-2\widetilde f}
\]
and
\[ x=1-\widetilde f-\widetilde\lambda_k,
\]
we have
\[ 1-\widetilde f-\sqrt{1-2\widetilde f}
 \le x\le
 1-\widetilde f+\sqrt{1-2\widetilde f}.
\]
The right-hand side of
\eqref{eq:positive-numerator} is a concave quadratic polynomial in
$x$, so its minimum on the above interval occurs at two endpoints.
Set 
\[  d:=\sqrt{1-2\widetilde f},\]
The two endpoints are
\[\begin{aligned}
x_{\pm}&=1-\widetilde{f}\pm d\\
&=\frac{1+d^2}{2}\pm d\\
&=\frac{(1\pm d)^2}{2}.
\end{aligned}\]
and \[\begin{aligned}
L &=n-1-n\widetilde{f}\\
&=n-1-\frac{n}{2}(1-d^2)\\
&=\frac{n-2+nd^2}{2}.
\end{aligned}
\]
Substitution this into \eqref{eq:positive-numerator} gives
\begin{align*}
 &H(x_\pm):=n\widetilde f^2+4x_\pm L-2(n-3)x_\pm^2\\
 &\quad=
 \frac{(1\pm d)^2}{4}
 \left[
 (3n-2)\mp2(3n-6)d+(3n+6)d^2
 \right].
\end{align*}
At $x_{-}$,  $H(x_{-})>0$ follows from $n\geq 4$ and $d>0$.  At $x_{+}$,  
the discriminant of the quadratic polynomial
$(3n+6)d^2-2(3n-6)d+(3n-2)$ with respect to $d$ is
\[
\Delta=4(3n-6)^2-4(3n+6)(3n-2)=192(1-n)<0,
\]
which implies that $H(x_{+})>0$.
Consequently, we prove that 
\[ 3N-D>0,\]
which implies that 
$$Q_k(X)\geq \frac{3N-D}{D}t^2>0.$$
\medskip
\noindent
\textbf{Step 4. Uniform estimates for large and bounded trace.}

We first consider the region where $\sigma_1(\lambda)$ is large. 
We claim that there exists some constant $S>0$ depending only on $n$, $K$, $\overline{f}$ and $\underline{f}$
such that when $\sigma_1(\lambda)>S$, we have 
\begin{equation}\label{eq:large-trace}
Q_k(X)\ge \frac{7}{6}\frac{g_k}{\sigma_1(\lambda)} t^2.
\end{equation}
For simplify, we set 
$$m=\frac{3N-D}{D},\qquad \widetilde g_k=\frac{g_k}{\sigma_1(\lambda)}.$$ 
Form (\ref{eq22}), we need to prove $$m\geq \frac{7}{6}\widetilde g_k,$$ 
which yields (\ref{eq:large-trace}).

We assume that there is a sequence of $\lambda^{(j)}\in \Gamma_2$ satisfying $\lambda^{(j)}_i\geq -K$ and  $\sigma_2(\lambda^{(j)})/\sigma_1(\lambda^{(j)})=f^{(j)}$ with $\underline{f}\leq f^{(j)}\leq \overline{f}$ such that $$\sigma_1(\lambda^{(j)}) \rightarrow \infty\ \ \ \mbox{as}\ \ \ j\to\infty.$$
We normalize  $\lambda^{(j)}$ by setting 
\[
 \widetilde\lambda_i^{(j)}:=\frac{\lambda_i^{(j)}}{\sigma_1(\lambda^{(j)})},
 \qquad
 \widetilde f^{(j)}:=\frac{f^{(j)}}{\sigma_1(\lambda^{(j)})},
 \qquad
 \widetilde g_i^{(j)}:=\frac{g_i^{(j)}}{\sigma_1(\lambda^{(j)})}
 =1-\widetilde f^{(j)}-\widetilde\lambda_i^{(j)}.
\]
Clearly, we have $$\sum_i  \widetilde\lambda_i^{(j)}=1,\qquad 
\widetilde\lambda_i^{(j)}
=\frac{\lambda_i^{(j)}}{\sigma_1(\lambda^{(j)})}
\ge -\frac{K}{\sigma_1(\lambda^{(j)})}, \qquad \sigma_2(\widetilde \lambda^{(j)})=\widetilde f^{(j)}.
$$
Since $|\widetilde\lambda^{(j)}|\leq 1$ and $\underline{f}\leq f^{(j)}\leq \overline{f}$,
up to a subsequence, we assume
$$\widetilde\lambda^{(j)} \rightarrow \widetilde\lambda^{(\infty)},\qquad  \widetilde f^{(j)} \rightarrow 0 \ \ \ \mbox{as}\ \ \ j\to\infty.$$
Then 
$$\sum_i \widetilde\lambda_i^{(\infty)}=1,\qquad \widetilde\lambda_i^{(\infty)}\geq 0, \qquad \sigma_2( \widetilde\lambda^{(\infty)})=0.$$
Moreover, we can conclude that there exists some $p\in \{1,2,\ldots,n\}$ such that 
$$\widetilde\lambda^{(\infty)}=e_p=(0,0,\ldots,1,\ldots,0).$$

\medskip
\noindent
\textbf{Case 1: \(p\neq k\).}
In this case,
\[
\widetilde{\lambda}_k^{(j)}\longrightarrow 0,
\qquad
\widetilde f^{(j)}\longrightarrow 0.
\]
Hence
\[\widetilde g_k^{(j)}
 =1-\widetilde f^{(j)}-\widetilde\lambda_k^{(j)}
\longrightarrow 1,
\]
\[L^{(j)}=n-1-n\widetilde f^{(j)} \longrightarrow n-1,
\]
\[N^{(j)}
=n-1-2n\widetilde f^{(j)}+n(\widetilde f^{(j)})^{2}+2(\widetilde g_k^{(j)})^2
\longrightarrow n+1
\]
and 
\[D^{(j)} =(n+2)N^{(j)}-(L^{(j)}+2 \widetilde g_k^{(j)})^2 \longrightarrow n+1.
\]
Consequently, we have 
\[
3N^{(j)}-D^{(j)} \longrightarrow 2(n+1).
\]
Therefore,
\[
m^{(j)}=\frac{3N^{(j)}-D^{(j)}}{D^{(j)}}\longrightarrow 2,
\qquad
\frac{m^{(j)}}{\widetilde g_k^{(j)}}\longrightarrow 2.
\]

\medskip
\noindent
\textbf{Case 2: \(p=k\).}
In this case,
\[
\widetilde{\lambda}_k^{(j)}\longrightarrow 1,
\qquad
\widetilde f^{(j)}\longrightarrow 0.
\]
Hence
\[\widetilde g_k^{(j)}
 =1-\widetilde f^{(j)}-\widetilde\lambda_k^{(j)}
\longrightarrow 0,
\]

\[
L^{(j)}=n-1-n\widetilde f^{(j)} \longrightarrow n-1,
\]
\[
N^{(j)}
=n-1-2n\widetilde f^{(j)}+n(\widetilde f^{(j)})^{2}+2(\widetilde g_k^{(j)})^2 \longrightarrow n-1,
\]
and therefore
\begin{equation*}
\begin{aligned}
D^{(j)} =(n+2)N^{(j)}-(L^{(j)}+2 \widetilde g_k^{(j)})^2 \longrightarrow 3(n-1)>0.
\end{aligned}
\end{equation*}
Since \(\widetilde g_k^{(j)} >0\), we may write
\begin{equation*}
\frac{m^{(j)}}{\widetilde g_k^{(j)}}
= \frac{3N^{(j)}-D^{(j)}}{D^{(j)} \widetilde g_k^{(j)}}
= \frac{n (\widetilde f^{(j)})^{2}/\widetilde g_k^{(j)}+4L^{(j)}-2(n-3)\widetilde g_k^{(j)}
}{D^{(j)}}.
\end{equation*}
Observing that 
\[ \frac{n (\widetilde f^{(j)})^{2}}{\widetilde g_k^{(j)}}\ge0, 
\]
we have 
\begin{equation*}
\begin{aligned}
\liminf_{j\to\infty}\frac{m^{(j)}}{\widetilde g_k^{(j)}}
&\ge \frac{4(n-1)}{3(n-1)}=\frac43.
\end{aligned}
\end{equation*}
Consequently, in both case, we conclude that 
\begin{equation*}
\liminf_{j\to\infty}\frac{m^{(j)}}{\widetilde g_k^{(j)}}
\ge \frac{4(n-1)}{3(n-1)}=\frac43.
\end{equation*}
This implies that there exists some constant $S>0$ depending only on $n$, $K$, $\overline{f}$ and $\underline{f}$
such that whenever $\sigma_1(\lambda)>S$, we have 
$$m\geq \frac{7}{6}\widetilde g_k,$$
which prove the claim.

We next consider the compact region
\begin{equation}\label{eq23}
  \frac{2n}{n-1}\underline{f}\le \sigma_1(\lambda)\le S.
\end{equation}
We claim that there exists some constant $C_0>0$ depending only on $n$, $K$, $\overline{f}$ and $\underline{f}$ such that when $\sigma_1(\lambda)$ satisfies (\ref{eq23}) and $X$ satisfies $g\cdot X=0$, we have 
\begin{equation}\label{eq29}
Q_k(X)\ge C_0\frac{g_k}{\sigma_1(\lambda)} t^2.
\end{equation}
Clearly, when $\sum_i X_i=0$, this claim holds. 

To obtain a contradiction, suppose that there are sequences of $\lambda^{(j)}\in \Gamma_2$ and $X^{(j)}\in \mathbb{R}^n$
satisfying $\lambda^{(j)}_i\geq -K$ and  $\sigma_2(\lambda^{(j)})/\sigma_1(\lambda^{(j)})=f^{(j)}$ with $\underline{f}\leq f^{(j)}\leq \overline{f}$ 
such that $$ \frac{2n}{n-1}\underline{f}\le \sigma_1(\lambda^{(j)})\le S,\qquad g^{(j)}\cdot X^{(j)}=0,\qquad \sum_i X_i^{(j)} \neq 0,$$
but we have 
$$\frac{Q_{k}\bigl(X^{(j)}\bigr) }{
\left(g_{k}^{(j)}/\sigma_1(\lambda^{(j)})\right)
\left(\sum_i X_i^{(j)}\right)^2
}\longrightarrow 0 \ \ \ \mbox{as}\ \ \ j\to\infty. $$
Since both the constraint and \(Q_k\) are homogeneous with respect to \(X\), we may normalize 
$$\sum_i X_i^{(j)}=1.$$
In fact, (\ref{eq24}) and (\ref{eq:g-upper}) give the uniformly bound of $g_{k}^{(j)}/\sigma_1(\lambda^{(j)})$. It follows  that
$$Q_{k}\bigl(X^{(j)}\bigr) \longrightarrow 0 \ \ \ \mbox{as}\ \ \ j\to\infty.$$
This implies that \(X^{(j)}\) is bounded, which follows from 
\[\begin{aligned}
\left(X_k^{(j)}\right)^2+3\sum_{i\ne k}\left(X_i^{(j)}\right)^2
&=Q_k\bigl(X^{(j)}\bigr)
+
\left(\sum_iX_i^{(j)}\right)^2\\
&=Q_k\bigl(X^{(j)}\bigr)+1.
\end{aligned}
\]
Therefore, passing to a subsequence, we may assume that
$$X^{(j)}\longrightarrow X^\infty \ \ \ \mbox{as}\ \ \ j\to\infty.$$

By the assumption (\ref{eq23}) and the two inequalities $|\lambda^{(j)}|\leq \sigma_1(\lambda^{(j)})$ and $|g^{(j)}|\leq 2 \sqrt{n} \sigma_1(\lambda^{(j)})$, we know that $\lambda^{(j)}$, $g^{(j)}$, $\sigma_1(\lambda^{(j)})$ and $f^{(j)}$ are all bounded. Hence, up to a subsequence, we may also assume that
\[\begin{aligned}
\lambda^{(j)}&\longrightarrow\lambda^\infty,
&\qquad
\sigma_1(\lambda^{(j)})&\longrightarrow\sigma_1(\lambda^\infty),\\
f^{(j)}&\longrightarrow f^\infty,
&
g^{(j)}&\longrightarrow g^\infty
\qquad\mbox{as }\ \ j\to\infty.
\end{aligned}
\]
which satisfy 
$$\sum_i X_i^\infty=1,
\qquad
g^\infty\cdot X^\infty=0,
\qquad
Q_k(X^\infty)=0.
$$
In addition, we let $D^{\infty}$ be the limit of the sequence of $D^{(j)}$. If \(D^\infty>0\), then strict positivity for $Q_k$ proved in Step 3 gives
\[ Q_k(X^\infty) >0,
\]
which contradicts \(Q_k(X^\infty)=0\).

If $D^{\infty}=0$, then we know that   all $ g_i^{\infty}$ are equal, i.e., $ g_i^{\infty}=c$ for some constant $c$. Hence the homogeneous constraint is $$0=\sum_i g_i^{\infty}X_i^{\infty}=c \sum_i X_i^{\infty}=c.$$
Due to $  g_i^{\infty}>0$, we have $c>0$, which yields a contradiction.

As a consequence, we complete the proof of the desired claim.

It remains to introduce the shifted quantity \(a>0\) depending only on  $n$, $K$, $\overline{f}$ and $\underline{f}$. Choose $\theta=\frac{1}{36}$ such that
\[ \frac{10}{9}=1+4\theta<\frac76,
\]
and choose
\begin{equation*}
 a\ge  S\max\left\{ 0,\frac{10}{9C_0}-1
 \right\}.
\end{equation*}
If $\sigma_1(\lambda)\ge S$, then we have 
\begin{equation*}
  \begin{split}
    Q_k(X)&\ge \frac{7}{6}\frac{g_k}{\sigma_1(\lambda)} t^2  \\
      & \geq \frac{10}{9} \frac{g_k}{\sigma_1(\lambda)} t^2\\
      &\geq \frac{10}{9}\frac{g_k}{\sigma_1(\lambda)+a} t^2.
  \end{split}
\end{equation*}
Since \(\sigma_1(\lambda)\le S\), this choice of $a$ implies that 
\[ a\ge
\sigma_1(\lambda) \left(\frac{10}{9C_0}-1\right),
\]and therefore
\[
C_0(\sigma_1(\lambda)+a)\ge\frac{10}{9}\sigma_1(\lambda),
\] that is,
\[
\frac{C_0}{\sigma_1(\lambda)}
\ge
\frac{10}{9} \frac{1}{\sigma_1(\lambda)+a}.
\]
Combining this with \eqref{eq29}, we conclude that
\[
Q_k(X)\ge  C_0\frac{g_k}{\sigma_1(\lambda)} t^2\geq 
\frac{10}{9}\frac{g_k}{\sigma_1(\lambda)+a}t^2.
\]
It follows that we prove 
\begin{equation}\label{eq27}
  Q_k(X)\geq 
\frac{10}{9}\frac{g_k}{\sigma_1(\lambda)+a}t^2
\end{equation}
under the homogeneous constraint $ q=\sum_i g_iX_i=0.$

\medskip
\noindent
\textbf{Step 5. Coercivity on the homogeneous hyperplane.}

For $g\cdot Y=\sum_i g_iY_i=0$, 
we  establish a crude coercivity estimate
\begin{equation}\label{eq:homogeneous-coercivity}
  Q_k(Y)\geq C_* \left(\frac{f}{\sigma_1(\lambda)}\right)^2 |Y|^2
\end{equation}
for some constant $C_*>0$ depending only on   $n$, $K$, $\overline{f}$ and $\underline{f}$.

Indeed, if $|\sum_i Y_i|\le|Y|/2$, then  we obtain 
$$Q_k(Y) = Y W Y^T-\left(\sum_iY_i\right)^2 \geq |Y|^2-\frac{1}{4}|Y|^2 =\frac{3}{4}|Y|^2.$$
Recalling\[
0<\frac{f}{\sigma_1(\lambda)}\leq\frac{n-1}{2n}<1,
\] we have 
$$Q_k(Y)\geq \frac{3}{4} \left(\frac{f}{\sigma_1(\lambda)}\right)^2 |Y|^2.$$
Otherwise, if  $|\sum_i Y_i| > |Y|/2$ , by (\ref{eq27}), we have 
\begin{equation}\label{eq26}
  \begin{aligned}
Q_k(Y) &\geq
\frac{10}{9}\frac{g_k}{\sigma_1(\lambda)+a}\left(\sum_iY_i\right)^2\\
&\geq \frac{10}{36}\frac{g_k}{\sigma_1(\lambda)+a} |Y|^2.
\end{aligned}
\end{equation}
By the aid of (\ref{eq24}) and the increase of $s/(s+a)$ with respect to $s$, we discover 
\[
\begin{aligned}
\frac{g_k}{\sigma_1(\lambda)+a}
&= \frac{g_k}{\sigma_1(\lambda)} \frac{\sigma_1(\lambda)}{\sigma_1(\lambda)+a}\\
&\geq \frac{n}{2(n-1)}
\frac{s_0}{s_0+a}
\left(\frac{f}{\sigma_1(\lambda)}\right)^2,
\end{aligned}
\]
where $s_0=\frac{2 n}{n-1}\underline{f}$ follows from (\ref{eq25}).
 Substituting this into (\ref{eq26}) yields
 \[
Q_k(Y) \geq
\frac{10}{72}\frac{n}{n-1}\frac{s_0}{s_0+a}
\left(\frac{f}{\sigma_1(\lambda)}\right)^2 |Y|^2.
\]
Taking \[C_*:=
\min\left\{\frac{3}{4}, \frac{10}{72}\frac{n}{n-1}\frac{s_0}{s_0+a}\right\}
>0,
\]
we prove \eqref{eq:homogeneous-coercivity}.

Taking into consideration the sum of (\ref{eq27}) multiplied by $\frac{39}{40}$  and (\ref{eq:homogeneous-coercivity}) multiplied by $\frac{1}{40}$, we obtain 
\begin{equation}\label{eq:combined-homogeneous}
 Q_k(Y)
 \ge
\frac{13}{12}\frac{g_k}{\sigma_1(\lambda)+a}
 \left(\sum_iY_i\right)^2
 +  \frac{C_*}{40}\left(\frac{f}{\sigma_1(\lambda)}\right)^2|Y|^2.
\end{equation}

\medskip
\noindent
\textbf{Step 6. Removal of the nonhomogeneous constraint.}

We  consider   the nonhomogeneous constraint
$$q=\sum_i g_iX_i\neq 0.$$
In order to remove the nonhomogeneous constraint, we set
$$Z:=\frac{g\cdot X}{|g|^2}g,
\qquad Y:=X-Z. $$
Then $$
g\cdot Y = g\cdot X-\frac{g\cdot X}{|g|^2}|g|^2=0.
$$
 By \eqref{eq28},  the Cauchy–Schwarz inequality leads to 
\begin{equation}\label{eq:Z-bound}
  \left|\sum_i Z_i\right|\leq \sqrt{n}|Z| =\sqrt{n} \frac{|q|}{|g|} \leq \frac{2 n}{n-1} \frac{|q|}{\sigma_1(\lambda)}.
\end{equation}
Expanding the quadratic form $Q_k$ at $Y+Z$ yields that 
$$Q_k(X) = Q_k(Y) + 2B_k(Y,Z)+Q_k(Z),$$
where $$B_k(Y,Z) = Y W Z^{T} -\left(\sum_i Y_i\right)\left(\sum_i Z_i\right).$$
Since $W_k$ has eigenvalues  $1$ and $3$, it is clear that 
$$|B_k(Y,Z)| \leq (n+3) |Y| |Z|$$ and 
$$Q_k(Z) \geq -(n-1)|Z|^2.$$
This with \eqref{eq:combined-homogeneous} leads to the estimate 
\begin{equation*}
  \begin{split}
    Q_k(X)  & \geq \frac{13}{12}\frac{g_k}{\sigma_1(\lambda)+a}
 \left(\sum_iY_i\right)^2
 +  \frac{C_*}{40}\left(\frac{f}{\sigma_1(\lambda)}\right)^2|Y|^2 \\
      & - 2(n+3) |Y| |Z|-(n-1) |Z|^2.
  \end{split}
\end{equation*}
By the Young's inequality
\begin{equation*}
  \begin{split}
    2(n+3)|Y| |Z|
= & \left(\sqrt{ \frac{C_*}{40}} \frac{f}{\sigma_1(\lambda)}|Y|\right)
\left(\frac{2(n+3)\sqrt{40}}{\sqrt{C_*}} \frac{\sigma_1(\lambda)}{f}|Z|\right)\\
      \leq &  \frac{C_*}{80} \left(\frac{f}{\sigma_1(\lambda)}\right)^2
|Y|^2+\frac{160(n+3)^2}{2C_*}\left(\frac{\sigma_1(\lambda)}{f}\right)^2
|Z|^2,
  \end{split}
\end{equation*}
we arrive at the estimate
\begin{equation*}
  \begin{split}
    Q_k(X)  & \geq \frac{13}{12}\frac{g_k}{\sigma_1(\lambda)+a}
 \left(\sum_iY_i\right)^2
 +  \frac{C_*}{80}\left(\frac{f}{\sigma_1(\lambda)}\right)^2|Y|^2 \\
      & - \frac{160(n+3)^2}{2C_*}\left(\frac{\sigma_1(\lambda)}{f}\right)^2
|Z|^2-(n-1) |Z|^2.
  \end{split}
\end{equation*}
Combining this with $0<f/\sigma_1(\lambda)<1$ and the estimate \eqref{eq:Z-bound} of $Z$, we deduce that 
\begin{equation*}
  \begin{split}
    Q_k(X)  & \geq \frac{13}{12}\frac{g_k}{\sigma_1(\lambda)+a}
 \left(\sum_iY_i\right)^2
 +  \frac{C_*}{80}\left(\frac{f}{\sigma_1(\lambda)}\right)^2|Y|^2 \\
      & - \frac{160(n+3)^2}{2C_*}\left(\frac{\sigma_1(\lambda)}{f}\right)^2
|Z|^2-(n-1) \left(\frac{\sigma_1(\lambda)}{f}\right)^2|Z|^2\\
& \geq \frac{13}{12}\frac{g_k}{\sigma_1(\lambda)+a}
 \left(\sum_iY_i\right)^2
 +  \frac{C_*}{80}\left(\frac{f}{\sigma_1(\lambda)}\right)^2|Y|^2 \\
 &-\frac{4n}{(n-1)^2}\left(\frac{160(n+3)^2}{2C_*}+n-1\right) \frac{q^2}{f^2}.
  \end{split}
\end{equation*}
Consequently, there exists some constant $C>0$ depending only on  $n$, $K$, $\overline{f}$ and $\underline{f}$ such that 
\begin{equation*}
 Q_k(X)  \ge
 \frac{13}{12}\frac{g_k}{\sigma_1(\lambda)+a}
 \left(\sum_i Y_i\right)^2
 -Cf^{-2}q^2.
\end{equation*}

Finally, we will replace the term $\left(\sum_i Y_i\right)^2$ by $\left(\sum_i X_i\right)^2$. 
By the aid of the Young’s inequality with $\epsilon=\frac{2}{39}$, we have 
$$\left(\sum_i X_i-\sum_i Z_i\right)^2\geq \frac{37}{39}\left(\sum_i X_i\right)^2-\frac{37}{2} \left(\sum_i Z_i\right)^2.$$
Furthermore, since $0<g_k<2 \sigma_1(\lambda)$, $0<f/\sigma_1(\lambda)<1$ and \eqref{eq:Z-bound}, we have 
\[
\begin{aligned}
\frac{g_k}{\sigma_1(\lambda)+a}\left(\sum_i Z_i\right)^2
&\leq \frac{2\sigma_1(\lambda)}{\sigma_1(\lambda)+a}
       \left(\sum_i Z_i\right)^2\\
&\leq \frac{8n^2}{(n-1)^2}\frac{q^2}{\sigma_1^2(\lambda)}\\
&\leq \frac{8n^2}{(n-1)^2}\frac{q^2}{f^2}.
\end{aligned}
\]
It follows that 
\begin{equation*}
  \begin{split}
    Q_k(X)  &\ge
 \frac{13}{12}\frac{g_k}{\sigma_1(\lambda)+a}
 \left(\sum_i Y_i\right)^2
 -Cf^{-2}q^2\\
      & \geq \frac{37}{36}\frac{g_k}{\sigma_1(\lambda)+a}
 \left(\sum_i X_i\right)^2-\frac{481}{24}\frac{g_k}{\sigma_1(\lambda)+a}
 \left(\sum_i Z_i\right)^2-Cf^{-2}q^2\\
 &\geq \frac{37}{36} \frac{g_k}{\sigma_1(\lambda)+a}
 \left(\sum_i X_i\right)^2-\left(\frac{481}{3}\frac{n^2}{(n-1)^2}+C\right)f^{-2}q^2.
  \end{split}
\end{equation*}
Summing this inequality over $k$ and using \eqref{eq:tensor-reduction-semiconvex} yields the desired result.
\end{proof}
\bigskip

Consider the equation $$G(D^2u,x):=\sigma_2(D^2u)-f\sigma_1(D^2u)=0,$$
whose linearized operator is given by
$$ G_{ij}:=\frac{\partial G}{\partial u_{ij}}=(\Delta u-f)\delta_{ij}-u_{ij}$$
and set $G=(G_{ij})$.
Differentiating the equation once gives
\begin{equation}\label{eq:first-diff}
 G_{ij}u_{ijk}=\Delta u f_k.
\end{equation}
Differentiating once more, we have 
  $$G_{ij}u_{ijkk}+\partial_k G_{ij} u_{ijk}=f_{kk} \Delta u+f_k \partial_k \Delta u=f_{kk} \Delta u+f_k \triangle u_k.$$ 
Due to $$\partial_k G_{ij}=\partial_k (\Delta u\,\delta_{ij}-u_{ij}-f\delta_{ij})=(\Delta u)_k \delta_{ij}-u_{ijk}-f_k\delta_{ij},$$
we have 
\begin{equation*}
  \begin{split}
     \partial_k G_{ij} u_{ijk}=& (\Delta u_k \delta_{ij}-u_{ijk}-f_k\delta_{ij})u_{ijk} \\
       =&\Delta u_k u_{iik}-u_{ijk}^2-f_k u_{iik}\\
       =& \Delta u_k \Delta u_k-u_{ijk}^2-f_k \Delta u_k.
   \end{split}
\end{equation*}
  Hence
\begin{equation}\label{eq:second-diff}
 \begin{split}
   G_{ij}(\Delta u)_{ij} & =|D^3u|^2-|D \Delta u|^2+\Delta u\Delta f+2Df\cdot D \Delta u\\
     & =u_{ijk}^2- (\Delta u_k)^2+\Delta f \Delta u +2 f_k \Delta u_k.
 \end{split}
\end{equation}
Now we apply Lemma \ref{lem:algebra-semiconvex} to obtain  the shifted trace Jacobi inequality.

\begin{proposition}\label{prop:strong-jacobi}
Let $f\in C^{2}(B_{10})$ with $\inf_{B_{10}}f>0$ and $u\in C^4(B_{10})$ be a 2-convex solution of (\ref{eq:main-equation}) with $D^2u\geq -KI$ for some positive constant $K$. Then there are some constants $a>0$ and $C>0$ depending only on $n$, $K$, $\sup_{B_{10}}f$, $\inf_{B_{10}}f$ and $\|Df\|_{L^{\infty}(B_{10})}$ such that 
\begin{equation}\label{eq:strong-jacobi}
 G_{ij}\beta_{ij}\geq \frac{1}{72} G_{ij}\beta_i\beta_j+h  \Delta f-C\Delta u,
\end{equation}
where \[
 \beta:=\log (\Delta u+a),\qquad h:=\frac{\Delta u}{\Delta u+a}.
\]
\end{proposition}

\begin{proof}
For any fixed $x_0\in B_{10}$, we choose an orthogonal matrix $Q$ such that
$$Q^T D^2u(x_0)Q=diag(\lambda_1,\ldots,\lambda_n).$$
Then we consider the transformation 
$$x=Qy, \qquad \tilde{u}(y)=u(Qy), \qquad \tilde{f}(y)=f(Qy).$$
Moreover, we define new equation $\tilde{G}$ as $$\tilde{G}(D^2 \tilde{u},y):=\sigma_2(D^2\tilde{u})-\tilde{f}\sigma_1(D^2\tilde{u})=0$$
and  the  linearized operator $\tilde{G}=Q^{T}GQ$ is also diagonal at $Q^{-1}x_0$. It is easy to check that 
\[
\begin{aligned}
\tilde{G}_{ij}(\Delta\tilde{u})_i(\Delta\tilde{u})_j
 &=G_{ij}(\Delta u)_i(\Delta u)_j,\\
|Df|&=|D\tilde{f}|, \qquad |D^k u|=|D^k\tilde{u}|\quad (k=1,2,3).
\end{aligned}
\]
Without loss of generality, we assume that $D^2u$ is  diagonal.
Apply Lemma \ref{lem:algebra-semiconvex} to $T_{ijk}=u_{ijk}$, 
$t_k=\Delta u_k$ and  $q_k=\Delta uf_k$ from \eqref{eq:first-diff}, we get
\begin{equation*}
 |D^3u|^2-|D \Delta u|^2
 \ge\frac{37}{36}\frac{1}{\Delta u+a} G_{ij} \Delta u_i\Delta u_j
      -C\frac{(\Delta u)^2}{f^2}|Df|^2,
\end{equation*}
for some constants $a>0$ and $C>0$ depending only on $n$, $K$, $\sup_{B_{10}}f$ and $\inf_{B_{10}}f$. Then, by (\ref{eq:second-diff}), we get
\begin{equation}\label{eq7}
  G_{ij}(\Delta u)_{ij}
 \ge\frac{37}{36}\frac{G_{ij} \Delta u_i\Delta u_j}{\Delta u+a} 
      -C\frac{(\Delta u)^2}{f^2}|Df|^2+\Delta u\Delta f+2Df\cdot D \Delta u.
\end{equation}
Since
$$\beta_i=\frac{(\Delta u)_i}{\Delta u+a}$$
and 
\[
\beta_{ij}
 =\frac{1}{\Delta u+a}(\Delta u)_{ij}-\frac{1}{(\Delta u+a)^2}(\Delta u)_i(\Delta u)_j,
\]
by substituting this into  (\ref{eq7}), we obtain
\begin{equation*}
\begin{aligned}
G_{ij}\beta_{ij}
&\geq \frac{1}{36}
 \frac{G_{ij}(\Delta u)_i(\Delta u)_j}{(\Delta u+a)^2}
 -C\frac{(\Delta u)^2}{f^2(\Delta u+a)}|Df|^2\\
&\qquad
 +\frac{\Delta u}{\Delta u+a}\Delta f
 +\frac{2}{\Delta u+a}Df\cdot D\Delta u\\
&\geq \frac{1}{36}G_{ij}\beta_i\beta_j+h\Delta f
 +2Df\cdot D\beta
 -C\frac{(\Delta u)^2}{f^2(\Delta u+a)}|Df|^2.
\end{aligned}
\end{equation*}
By  the Young’s inequality with $\epsilon=\frac{1}{72}$ and $ \tr(G^{-1})\le 2(n-1)\frac{\Delta u}{f^2}$, we have
\begin{equation*}
\begin{aligned}
2|Df\cdot D\beta|
&=2|G^{-1/2}Df\cdot G^{1/2}D\beta|\\
&\leq \frac{1}{72}G_{ij}\beta_i\beta_j
 +72(G^{-1})_{ij}f_i f_j\\
&\leq \frac{1}{72}G_{ij}\beta_i\beta_j
 +\frac{144(n-1)\Delta u}{f^2}|Df|^2\\
&\leq \frac{1}{72}G_{ij}\beta_i\beta_j
 +\frac{144(n-1)}{\inf_{B_{10}}f^2}
  \|Df\|_{L^{\infty}(B_{10})}^2\Delta u.
\end{aligned}
\end{equation*}
Moreover, using $$\frac{(\Delta u)^2}{f^2(\Delta u+a)}|Df|^2\leq \frac{1}{\inf_{B_{10}}f^2}\|Df\|_{L^{\infty}(B_{10})}^2 \Delta u,$$ we obtain
$$ G_{ij}\beta_{ij}\geq  \frac{1}{72} G_{ij}\beta_i\beta_j+h  \Delta f-\left(\frac{144(n-1) }{ \inf_{B_{10}}f^2}+\frac{C}{\inf_{B_{10}}f^2}\right)\|Df\|_{L^{\infty}(B_{10})}^2 \Delta u.$$
\end{proof}

\begin{proposition}[Weak Jacobi inequality]\label{prop:weak-jacobi}
Let $f\in C^{2}(B_{10})$ with $\inf_{B_{10}}f>0$ and $u\in C^4(B_{10})$ be a 2-convex solution of (\ref{eq:main-equation}) with $D^2u\geq -KI$ for some positive constant $K$. Then there exists some constant $C>0$ depending only on $n$, $K$, $\sup_{B_{10}}f$, $\inf_{B_{10}}f$ and $\|Df\|_{L^{\infty}(B_{10})}$ such that 
\begin{equation*}
\diver(GD\beta)\geq \frac{1}{144} G_{ij}\beta_i\beta_j+\diver(hDf)-C\Delta u
\end{equation*}
in the sense of distributions, where $\beta$ and $h$ are defined in Proposition \ref{prop:strong-jacobi}.
\end{proposition}

\begin{proof}
Since $\partial_iG_{ij}=-f_j$ and $h_i=(1-h)\beta_i$, we have
\[
 \diver(GD\beta)=D_i(G_{ij}\beta_j)=G_{ij}\beta_{ij}-f_j\beta_j
\]
and $$h\Delta f=\diver(hDf)-Dh\cdot Df=\diver(hDf)-(1-h)Df\cdot D\beta.
$$
Inserting this into Jacobi inequality (\ref{eq:strong-jacobi}) leads to 
$$  
\diver(GD\beta)\geq \frac{1}{72} G_{ij}\beta_i\beta_j+\diver(hDf)-C\Delta u-(2-h)Df \cdot D\beta.$$
Using the Young’s inequality with $\epsilon=\frac{1}{144}$, $ \tr(G^{-1})\le 2(n-1)\frac{\Delta u}{f^2}$ and $1\leq 2-h<2$, we obtain
\begin{equation*}
  \begin{split}
    2|Df \cdot D\beta|&\leq  \frac{1}{144} G_{ij}\beta_i\beta_j+144(G^{-1})_{ij}f_if_j \\
      & \leq \frac{1}{144} G_{ij}\beta_i\beta_j
 +288\frac{(n-1) \Delta u}{ f^2}|Df|^2
  \end{split}
\end{equation*}
and therefore
$$  \diver(GD\beta)\geq \frac{1}{144} G_{ij}\beta_i\beta_j+\diver(hDf)-C\Delta u,$$
where $C>0$ only depends on $n$, $K$, $\sup_{B_{10}}f$, $\inf_{B_{10}}f$ and $\|Df\|_{L^{\infty}(B_{10})}$.
Equivalently, for every nonnegative function $\varphi\in C_c^1(B_{10})$, we have 
\begin{equation}\label{eq:weak-test}
 -\int_{B_{10}} G_{ij}\beta_j \varphi_i\,dx
 \ge \frac{1}{144}  \int_{B_{10}} G_{ij}\beta_i\beta_j \varphi\,dx
       -\int_{B_{10}}h f_i \varphi_i\,dx-C\int_{B_{10}}  \Delta u \varphi\,dx.
\end{equation}
\end{proof}

\begin{corollary}[Caccioppoli-type estimate]\label{cor:energy}
Let $f\in C^{2}(B_{10})$ with $\inf_{B_{10}}f>0$ and $u\in C^4(B_{10})$ be a 2-convex solution of (\ref{eq:main-equation}) with $D^2u\geq -KI$ for some positive constant $K$. If $0\le\Phi\le1$ and $\Phi\in C_c^1(B_2)$, then there exists some constant $C>0$ depending only on $n$, $K$, $\sup_{B_{10}}f$, $\inf_{B_{10}}f$, $\|Df\|_{L^\infty(B_{10})} $and $\|D\Phi\|_{L^\infty(B_{10})} $ such that 
\begin{equation*}
 \int_{B_2}\Phi^2G_{ij}\beta_i\beta_j\,dx
 \le C\int_{\operatorname{supp}\Phi} \Delta u\,dx,
\end{equation*}
where $\beta$ comes from Proposition \ref{prop:strong-jacobi}.
\end{corollary}

\begin{proof}

Take  $\varphi=\Phi^2$ as the test function in \eqref{eq:weak-test}. From $\varphi_i=2 \Phi \Phi_i$, it follows that
\[
\begin{aligned}
\frac{1}{144}\int_{B_2}\Phi^2G_{ij}\beta_i\beta_j\,dx
\leq &-2\int_{B_2}\Phi G_{ij}\beta_j\Phi_i\,dx\\
& +2\int_{B_2}h\Phi f_i\Phi_i\,dx
 +C\int_{B_2}\Delta u\,\Phi^2\,dx.
\end{aligned}
\]
The term $\Phi G_{ij}\beta_j \Phi_i$ can be absorbed by the left-hand term through the Young’s inequality with $\epsilon=\frac{1}{288}$ 
$$
2\left|\Phi G_{ij}\beta_j \Phi_i\right|
\leq \frac{1}{288} \Phi^2 G_{ij}\beta_i\beta_j +
288 G_{ij}\Phi_i\Phi_j.
$$
Recalling $0<G<2 \Delta u I$, we get 
$$288\int_{B_2}  G_{ij}\Phi_i\Phi_j\,dx\leq 576 \|D\Phi\|_{L^{\infty}(B_{10})}^2\int_{\operatorname{supp}\Phi} \Delta u\,dx.$$
Using (\ref{eq:s-lower}), we  see that 
\begin{equation*}
  \begin{split}
     2 \left|\int_{B_2}h \Phi f_i \Phi_i\,dx\right|& \leq 2 \frac{n-1}{2n \inf_{B_{10}}f}\left|\int_{B_2} \Delta u h \Phi f_i \Phi_i\,dx\right|  \\
      & \leq  \frac{n-1}{n \inf_{B_{10}}f}\|D\Phi\|_{L^{\infty}(B_{10})}\|Df\|_{L^{\infty}(B_{10})} \left|\int_{\operatorname{supp}\Phi}  \Delta u\,dx \right|.
  \end{split}
\end{equation*}
Taking the above estimates into consideration, we obtain
\[
\begin{aligned}
\frac{1}{288}\int_{B_2}\Phi^2G_{ij}\beta_i\beta_j\,dx
&\leq \left[
\begin{gathered}
576\|D\Phi\|_{L^{\infty}(B_{10})}^2+C\\
{}+\dfrac{n-1}{n\inf_{B_1}f}\|D\Phi\|_{L^{\infty}(B_{10})}\|Df\|_{L^{\infty}(B_{10})}
\end{gathered}
\right]\\
&\qquad\times
\int_{\operatorname{supp}\Phi}\Delta u\,dx.
\end{aligned}
\]
\end{proof}

\section{Legendre--Lewy transform and mean value inequality}
\noindent

After subtracting a linear function, assume $u(0)=0$ and $Du(0)=0$.  For $\tau=K+1$, we 
define
\begin{equation*}
 y=T(x)=\tau x+Du(x),
 \qquad J(x)=\det DT(x)=\det(\tau I+D^2u(x)).
\end{equation*}
Semi-convexity gives  $DT\ge I$.   
From \[ \bigl(T(x)-T(z)\bigr)\cdot(x-z)
\geq(\tau-K)|x-z|^2
=|x-z|^2,\]
it follows that $T$ is one-to-one and its inverse is
$1$-Lipschitz.    Moreover, for $|x|=1$,
\[ T(x)\cdot x=\tau +Du(x)\cdot x\ge1,
\]
which implies that 
\begin{equation*}
 B_1^y\subset T(B_1^x).
\end{equation*}

\begin{proposition}[Transformation rule]
Let $f\in C^{2}(B_{10})$ with $\inf_{B_{10}}f>0$ and $u\in C^4(B_{10})$ be a 2-convex solution of (\ref{eq:main-equation}) with $D^2u\geq -KI$ for some positive constant $K$. Then there exists some constant $C>0$ depending only on $n$, $K$, $\sup_{B_{10}}f$, $\inf_{B_{10}}f$ and $\|Df\|_{L^{\infty}(B_{10})}$ such that 
\begin{equation}\label{eq:pushed-jacobi}
 \diver(\cG D\beta^*)
 \ge\diver\mathcal F-C
 \qquad\text{in}\ \ \ B_1^y
\end{equation} 
in the sense of distributions, where 
$\beta^{*}(y):=\beta(T^{-1}y)$ and 
\[\mathcal{G}(y) :=\left(\frac{DTG (DT)^T}{J}\right)(T^{-1}y),
\]
\[\mathcal{F}(y):=\left(\frac{h DT Df}{J}\right)(T^{-1}y).
\]
Here $\beta$ and $h$ are defined in Proposition \ref{prop:strong-jacobi}.
\end{proposition}

\begin{proof}
For any  nonnegative function $\phi\in C^\infty_{c}(B_{1}^{y})$, we set
$$\varphi(x)=\phi(T(x)).$$
Since $ \operatorname{supp} \phi \Subset B_1^y\subset T(B_1^x)$ and $T$ is one-to-one, we conclude that $\varphi(x)$ is a legitimate test function in $B_1^x$. Plugging $\varphi$ in \eqref{eq:weak-test}, we have 
\begin{equation*}
  \begin{split}
   -\int_{B_1^x} G_{ij}\beta_j \varphi_i\,dx
 &\ge \frac{1}{144} \int_{B_1^x} G_{ij}\beta_i\beta_j \varphi\,dx
       -\int_{B_1^x}h f_i \varphi_i\,dx-C\int_{B_1^x}  \Delta u \varphi\,dx\\
   & \geq-\int_{B_1^x}h f_i \varphi_i\,dx-C\int_{B_1^x}  \Delta u \varphi\,dx.
  \end{split}
\end{equation*}
First, a direct computation yields 
$$D\varphi=(DT) ^{T} D\phi, \qquad D\beta=(DT) ^{T} D\beta^{*}.$$
Consequently, we have 
\[
\begin{aligned}
GD\beta\cdot D\varphi
&=G(DT)^TD\beta^*\cdot(DT)^TD\phi\\
&=D\phi^TDTG(DT)^TD\beta^*\\
&=\left(DTG(DT)^TD\beta^*\right)\cdot D\phi.
\end{aligned}
\]
By the coordinate transformation, we obtain 
$$
\int_{B_1^x} G D \beta\cdot D \varphi\,dx
= \int_{B_1^y} \frac{\left(DT G (DT)^{T}D\beta^{*}\right)}{J}
\cdot D\phi\,dy.
$$
Next, from $
Df\cdot D\varphi =Df\cdot (DT)^{T}D\phi=\left(DT Df\right)\cdot D\phi$, it follows that 
\[
\int_{B_1^x} h Df\cdot D\varphi\,dx =\int_{B_1^y} \frac{h DT Df}{J}\cdot D\phi\,dy.
\]
Finally, it is easy to see that 
\[ \int_{B_1^x} \Delta u\varphi\,dx = \int_{B_1^y} \frac{\Delta u}{J}\phi(y)\,dy.\] 
By $D^2u+KI>0$ and $J=\det DT=\det(I+KI+D^2u)$, we know that 
$$J\geq 1+nK+\Delta u.$$
Hence  
$$\int_{B_1^y} \frac{\Delta u}{J}\phi(y)\,dy\leq \int_{B_1^y} \phi(y)\,dy.$$
We complete the proof.
\end{proof}

We prove that the equation  \eqref{eq:pushed-jacobi} is uniformly elliptic. 

\begin{proposition}[Uniform ellipticity ]\label{prop:uniform}
Let $f\in C^{2}(B_{10})$ with $\inf_{B_{10}}f>0$ and $u\in C^4(B_{10})$ be a 2-convex solution of (\ref{eq:main-equation}) with $D^2u\geq -KI$ for some positive constant $K$. Then there exists some constant  $C>0$ depending only on
$n$, $K$, $\inf_{B_{10}}f$ and $\sup_{B_{10}}f$ such that
\begin{equation*}
 \frac{1}{C}I\le\cG\le C I,
 \qquad
 \|\mathcal F\|_{L^\infty(B_1^y)}\le \|Df\|_{L^{\infty}(B_1^x)}.
\end{equation*}
\end{proposition}

\begin{proof}
As in the proof of Proposition \ref{prop:strong-jacobi}, we can assume that $D^2u$ is diagonal at $x_0\in B_1$. Denote $\lambda_1\ge\cdots\ge\lambda_n\ge-K$ as the eigenvalues  of $D^2u(x_0)$. Then $DT=\tau I+D^2u$ and $G=(\Delta u-f)I-D^2u$ also are diagonal.

It is clear that 
$$\mathcal{G}_{ii}(y)=\left(\frac{DTG (DT)^T}{J}\right)(T^{-1}y)= \frac{(\tau+\lambda_i)^2 G_{ii}}
 {\prod_{j=1}^n(\tau+\lambda_j)}.$$
 In fact, Lemma \ref{lem:spectral} implies that all $\lambda_i$ for $i\geq 2$ satisfy
 $$1\leq \tau +\lambda_i\leq C$$
 for some constant $C>0$ depending only on $n$, $K$ and $\sup_{B_{10}}f$.

Lemma \ref{lem:spectral} gives
\[\frac{1}{C}\leq 
 \cG_{11}
 =\frac{(\tau+\lambda_1)G_{11}}
 {\prod_{j=2}^n(\tau+\lambda_j)}\leq C
\]
 and 
\[ \frac{1}{C}\leq 
 \cG_{ii}
 =\frac{(\tau+\lambda_i)^2G_{ii}}
 {(\tau+\lambda_1)\prod_{j=2}^n(1+\lambda_j)} \leq C, \qquad i\geq 2,
\]
where  constant $C>0$ depends only on $n$, $K$, $\inf_{B_{10}}f$ and $\sup_{B_{10}}f$.
This proves uniform ellipticity.  Finally,
\[
 \left\|\frac{DT}{J}\right\|
 =\max_i\frac{\tau+\lambda_i}{\prod_j(\tau+\lambda_j)}\le1,
\]
and therefore $\|\mathcal F\|_{L^{\infty}(B_1^y)}\le \|Df\|_{L^{\infty}(B_1^x)}$.
\end{proof}

We introduce the following  local boundedness estimate from \cite[Theorem 8.17]{GT83} and \cite{HL11}.

\begin{lemma}\label{lem:mean-value}
Let $F,E \in L^\infty(B_1)$ and $$\lambda I\leq A\leq \Lambda I$$
for some constants $0<\lambda\leq \Lambda<+\infty$.
 Assume that  $u\in W^{1,2}(B_1)$ satisfy
\[  \diver( A Du)\ge\diver F-E
\ \ \ \mbox{in}\ \ \ B_1.
\]
Then
\begin{equation*}
 \sup_{B_{1/2}}u
 \le C\left(1+\|u_+\|_{L^1(B_1)}
 +\|F\|_{L^\infty(B_1)}+\|E\|_{L^\infty(B_1)}\right),
\end{equation*}
where $C>0$ depends only on $n$, $\lambda$ and $\Lambda$.
\end{lemma}

Applying Lemma \ref{lem:mean-value} to \eqref{eq:pushed-jacobi} and using
Proposition \ref{prop:uniform}, we obtain
\begin{equation}\label{eq:beta-mean}
 \beta(0)=\beta^*(0)
 \le C\left(1+\int_{B_1^y}(\beta^*)_+\,dy \right)
\end{equation}
for some constant  $C>0$ depending only on $n$, $K$, $\sup_{B_{10}}f$, $\inf_{B_{10}}f$ and $\|Df\|_{L^{\infty}(B_{10})}$.

\section{Proof of Theorem \ref{thm:main}}
\noindent

Once we estimate the integral in \eqref{eq:beta-mean}, we prove Theorem \ref{thm:main}. 
In view of  Lemma \ref{lem:spectral} and $\lambda_1\leq \sigma_1(\lambda)$, 
the Jacobian $J$ satisfies
\begin{equation*}
 J=(K+1+\lambda_1)\prod_{i=2}^n(K+1+\lambda_i)
 \le C(K+1+\lambda_1)\leq C(1+\sigma_1(\lambda)),
\end{equation*}
where $\lambda=(\lambda_1,\ldots,\lambda_n)$ are the eigenvalues of $D^2u$.  
Hence
\begin{equation}\label{eq:beta-change}
 \begin{split}
  \int_{B_1^y}(\beta^*)_+(y)\,dy& =\int_{T^{-1}(B_1^y)} (\beta^*)_+(Tx)J(x)\,dx \\
     & =\int_{T^{-1}(B_1^y)} (\beta)_+(x)J(x)\,dx\\
     &\leq C \int_{B_1} \beta_+(1+\Delta u)\,dx,
 \end{split}
\end{equation}
where  $C>0$ depends only on $n$, $K$ and $\sup_{B_{10}}f$.

\begin{lemma}\label{lem:closure}
Let $f\in C^{2}(B_{10})$ with $\inf_{B_{10}}f>0$ and $u\in C^4(B_{10})$ be a 2-convex solution of (\ref{eq:main-equation}) with $D^2u\geq -KI$ for some positive constant $K$. Then we have 
\begin{equation}\label{eq:closure}
 \int_{B_1}\beta_+(1+\Delta u)\dd x
 \le C(1+\|Du\|_{L^{\infty}(B_2)})(\|Du\|_{L^{\infty}(B_2)}+a) 
\end{equation}
for some constant $C>0$ and $a>0$ depending only on  $n$, $K$, $\sup_{B_{10}}f$, $\inf_{B_{10}}f$ and $\|Df\|_{L^\infty(B_{10})}$, where $\beta$ is defined in Proposition \ref{prop:strong-jacobi}.
\end{lemma}

\begin{proof}
By divergence theorem, we have 
\begin{equation}\label{eq:s-L1}
 \int_{B_{2}} \Delta u\,dx=\int_{\partial B_{2} }Du\cdot \nu \,dS\leq C(n) \|Du\|_{L^{\infty}(B_2)}.
\end{equation}
Moreover, observing $\beta_{+}=(\log (\Delta u+a))_+\le \Delta u+a$, we get the estimate 
\begin{equation}\label{eq10}
  \int_{B_{2}}\beta_{+}\,dx\leq  \int_{B_{2}} (\Delta u+a)\,dx\leq C(n)( \|Du\|_{L^{\infty}(B_2)}+a).
\end{equation}

It remains to estimate $\int_{B_1} \beta_+ \Delta u\,dx$.  Let
$\chi\in C_c^1(B_{2})$ be $0\leq \chi \leq 1$ and $\chi=1$ in $B_{1}$.  Integration by parts gives
\[
\begin{aligned}
\int_{B_1}\chi\beta_+\Delta u\,dx
&\leq \int_{B_{2}}\chi\beta_+\Delta u\,dx\\
&=-\int_{B_{2}}\chi D\beta_+\cdot Du\,dx
  -\int_{B_{2}}\beta_+D\chi\cdot Du\,dx.
\end{aligned}
\]
Clearly, by (\ref{eq10}), we can estimate the second term 
\begin{equation*}
  \begin{split}
    \left|\int_{B_{2}} \beta_{+} D\chi\cdot Du\,dx\right| & \leq  \|Du\|_{L^{\infty}(B_2)} \|D\chi\|_{L^{\infty}(B_2)} \int_{B_{2}}\beta_{+}\,dx\\
       & \leq C(n)\|Du\|_{L^{\infty}(B_2)} \|D\chi\|_{L^{\infty}(B_2)} ( \|Du\|_{L^{\infty}(B_2)}+a).
   \end{split}
\end{equation*}
Using the Cauchy–Schwarz inequality and Hölder inequality, we see that 
\begin{equation*}
\begin{aligned}
\left|\int_{B_{2}}\chi D\beta_+\cdot Du\,dx\right|
&=\left|\int_{B_{2}}
 \chi G^{1/2}D\beta_+\cdot G^{-1/2}Du\,dx\right|\\
&\leq \int_{B_{2}}
 |\chi G^{1/2}D\beta_+|\,|G^{-1/2}Du|\,dx\\
&\leq
\left(\int_{B_{2}}
 \chi^2G_{ij}(\beta_+)_i(\beta_+)_j\,dx\right)^{1/2}\\
&\qquad\times
\left(\int_{B_{2}}(G^{-1})_{ij}u_i u_j\,dx\right)^{1/2}.
\end{aligned}
\end{equation*}
In view of Corollary \ref{cor:energy} and $D\beta_{+}=\mathbf{1}_{\{\beta>0\}}D\beta$ a.e., we have
\[\begin{aligned}
\int_{B_{2}}\chi^2G_{ij}(\beta_+)_i(\beta_+)_j\,dx
&\leq \int_{B_{2}}\chi^2G_{ij}\beta_i\beta_j\,dx\\
&\leq C\int_{B_{2}}\Delta u\,dx\\
&\leq C\|Du\|_{L^{\infty}(B_2)}.
\end{aligned}
\]
for some constant $C>0$ depending only on $n$, $K$, $\sup_{B_{10}}f$,
$\inf_{B_{10}}f$, $\|Df\|_{L^\infty(B_{10})}$ and
\linebreak[2]$\|D\chi\|_{L^\infty(B_{10})}$ .

For the second factor, (\ref{eq:G-basic}) in Lemma \ref{lem:shift} and \eqref{eq:s-L1} imply
\begin{equation*}
  \begin{split}
   \int_{B_{2}}(G^{-1})_{ij}u_iu_j\,dx  & \leq 2(n-1)\frac{\|Du\|_{L^{\infty}(B_2)}^2}{f^2}\int_{B_{2}} \Delta u\,dx \\
      & \leq \frac{C(n)}{\inf_{B_{10}}f^2}
\|Du\|_{L^{\infty}(B_2)}^3.
  \end{split}
\end{equation*}
This proves \eqref{eq:closure}.  
\end{proof}
\vspace{5pt}

\begin{proof}[Proof of Theorem \ref{thm:main}]
Combining \eqref{eq:beta-mean}, \eqref{eq:beta-change} and Lemma
\ref{lem:closure} gives \[
 \log (\Delta u(0)+a)\le  C(1+\|Du\|_{L^{\infty}(B_2)})(\|Du\|_{L^{\infty}(B_2)}+a),
\]
where  $C>0$ and $a>0$ depends only on  $n$, $K$, $\sup_{B_{10}}f$, $\inf_{B_{10}}f$ and $\|Df\|_{L^\infty(B_{10})}$.
Therefore, we have
$$|D^2u(0)|\leq \Delta u(0)\leq \exp{\{C(1+\|Du\|_{L^{\infty}(B_2)})(\|Du\|_{L^{\infty}(B_2)}+a)\}}.$$
Furthermore, applying the interior gradient estimate of Hessian quotient 
equations (see \cite[Theorem 1.1]{Chen15}), we obtain the estimate  \eqref{eq:main-estimate}.
\end{proof}

\section{Singular solution}
\noindent

\begin{proof}[Proof of Theorem \ref{thm2}]
Without loss of generality, we only give a proof in $n=2$.
We show that $u_\alpha$ is a convex viscosity solution of \eqref{eq:quotient} by smooth approximation. For $\varepsilon>0$, let $\phi_{\alpha,\varepsilon}$ be the smooth function
determined by
\begin{equation*}
 \phi_{\alpha,\varepsilon}''(t)
 =(t^2+\varepsilon^2)^{-\alpha/2},
 \qquad
 \phi_{\alpha,\varepsilon}(0)
 =\phi_{\alpha,\varepsilon}'(0)=0,
\end{equation*}
that is,
\[  \phi_{\alpha,\varepsilon}(t)
 =\int_0^t (t-s)(s^2+\varepsilon^2)^{-\alpha/2}\,ds.\]
 Set
\[ u_{\alpha,\varepsilon}(x_1,x_2)
 =\phi_{\alpha,\varepsilon}(x_1)+\frac{x_2^2}{2}.\]
Then $u_{\alpha,\varepsilon}$ is smooth and strictly convex, and
\[  D^2u_{\alpha,\varepsilon}
 =\operatorname{diag}  \bigl((x_1^2+\varepsilon^2)^{-\alpha/2},1\bigr).
\]
Consequently, it is a classical solution of
\begin{equation}\label{eq:smooth-equation}
 \frac{\sigma_2(D^2u_{\alpha,\varepsilon})}
      {\sigma_1(D^2u_{\alpha,\varepsilon})}
 =f_{\alpha,\varepsilon}(x),
 \qquad
 f_{\alpha,\varepsilon}(x)  =\frac{1}{1+(x_1^2+\varepsilon^2)^{\alpha/2}}.
\end{equation}
Since $|s|^{-\alpha}$ is locally integrable when $\alpha<1$, dominated
convergence in the integral representation of $\phi_{\alpha,\varepsilon}$ gives
\[ u_{\alpha,\varepsilon}\longrightarrow u_\alpha
 \quad\text{locally uniformly in }B_2.
\]
Also,
\[  f_{\alpha,\varepsilon}\longrightarrow f_\alpha
 \quad\text{uniformly in }B_2.
\]
 The standard stability theorem for viscosity solutions,
applied to \eqref{eq:smooth-equation}, therefore implies that $u_\alpha$ is a
convex viscosity solution of \eqref{eq:quotient} in $B_2$. 

At every point with $x_1\neq0$, we have 
\begin{equation*}
 D^2u_\alpha
 =\begin{pmatrix}
    |x_1|^{-\alpha}&0\\
    0&1
   \end{pmatrix},
\end{equation*}
which implies that 
\[ u_{\alpha,11}(x)=|x_1|^{-\alpha}
 \longrightarrow+\infty
 \qquad\text{as }\ x_1\to0,
\]
On the other hand, it is clear that $f_\alpha\in C^{0,\alpha}(B_2)$, but 
\[
 \frac{|f_\alpha(x_1,0)-f_\alpha(0,0)|}{|x_1|}
 =\frac{|x_1|^{\alpha-1}}{1+|x_1|^\alpha}
 \longrightarrow+\infty
 \qquad\text{as }\ x_1 \to0.
\]
Thus $f_\alpha$ is not Lipschitz at the origin. Consequently, the Lipschitz exponent in the $C^2$ regularity theorem is optimal.

\section*{Acknowledgments}
The authors are grateful to Xingchen Zhou for  his valuable comments and suggestions. The authors acknowledge the use
of AI tools. All mathematical arguments and proofs in the final manuscript were independently verified and written by the authors.

\end{proof}

\end{document}